\documentclass[10pt,a4paper,reqno]{amsart}
\usepackage{color}
\usepackage{verbatim}
\usepackage{mathtools}
\usepackage{amstext}
\usepackage{amsthm}
\usepackage{amssymb}
\usepackage{kotex}
\usepackage[unicode=true,bookmarks=false,breaklinks=
false,pdfborder={0 0 1},colorlinks=false]{hyperref}
\hypersetup{
 colorlinks,linkcolor={red!60!black},citecolor={green!60!black},urlcolor={blue!60!black}}

\makeatletter
\numberwithin{equation}{section}
\numberwithin{figure}{section}

\usepackage{lineno}
\usepackage{footnote}
\usepackage{enumitem}
\usepackage{setspace}
\usepackage{tikz}
\usepackage{cite}
\makesavenoteenv{tblr}

\theoremstyle{plain}
\newtheorem{theorem}{Theorem}[section]
\newtheorem{proposition}[theorem]{Proposition}
\newtheorem{lemma}[theorem]{Lemma}

\theoremstyle{remark}
\newtheorem{remark}[theorem]{Remark}

\global\long\def\R{\mathbf{\mathbb{R}}}%
\global\long\def\N{\mathbf{\mathbb{N}}}%
\global\long\def\Im{\mathrm{Im}}%
\global\long\def\Re{\mathrm{Re}}%
\global\long\def\jp#1{\langle#1\rangle}%
\global\long\def\norm#1{\|#1\|}%
\global\long\def\ol#1{\overline{#1}}%
\global\long\def\wt#1{\widehat{#1}}%
\global\long\def\td#1{\widetilde{#1}}%
\global\long\def\bbR{\mathbf{\mathbb{R}}}%
\global\long\def\bbC{\mathbf{\mathbb{C}}}%
\global\long\def\bbN{\mathbf{\mathbb{N}}}%
\global\long\def\calH{\mathcal{H}}%
\global\long\def\calO{\mathcal{O}}%
\global\long\def\eps{\varepsilon}%
\global\long\def\td#1{\widetilde{#1}}%
\global\long\def\NL{\textnormal{NL}}

\makeatother

\title[Log-log blow-up for HW]{A log-log upper bound on blow-up rates for the mass-critical half-wave equation}

\author{Taegyu Kim}
\email{k1216300@kias.re.kr}
\address{Korea Institute for Advanced Study, 80 Hoegi-ro, Dongdaemun-gu, Seoul 02455, Korea}

\author{Soonsik Kwon}
\email{soonsikk@kaist.edu}
\address{Department of Mathematical Sciences, Korea Advanced Institute of Science
and Technology, 291 Daehak-ro, Yuseong-gu, Daejeon 34141, Korea}

\author{Jeongheon Park}
\email{jse05002@kaist.ac.kr}
\address{Department of Mathematical Sciences, Korea Advanced Institute of Science
and Technology, 291 Daehak-ro, Yuseong-gu, Daejeon 34141, Korea}


\begin{document}

\begin{abstract}
    We study finite-time blow-up for the one-dimensional focusing mass-critical half-wave equation 
    \begin{equation*}
        i\partial_tu=|D|u-|u|^2u.
    \end{equation*}
    For even initial data with negative energy and mass slightly above the ground-state mass, we prove the log-log upper bound 
    \begin{equation*}
        \|u(t)\|_{\dot H^{1/2}}\lesssim \left(\frac{\log|\log(T-t)|}{T-t}\right)^{1/2}
        \quad \text{as}\quad t\uparrow T.
    \end{equation*}
    This gives, for the half-wave equation, the same log-log law upper bound as in the mass-critical nonlinear Schr\"odinger equation.
    
    The proof follows a similar strategy developed by Merle and Rapha\"el, but requires a new construction of the blow-up profile. Main difficulty arises from the nonlocal operator $|D|$ and the absence of pseudo-conformal symmetry. We construct an almost self-similar profile with exponentially small error by combining tail computations carried out to arbitrary order, depending on a dynamical parameter, with Borel integral summation in $\Lambda$-analytic spaces. Then, in the modulation analysis, we use a local-virial spectral property proved in the companion paper \cite{Park2026arXiv}.
\end{abstract}

\maketitle
\setcounter{tocdepth}{1}
\tableofcontents

\section{Introduction}
We consider the \emph{mass-critical one-dimensional focusing half-wave equation}
\begin{align}\label{eq:half-wave}
    \begin{cases}
        i\partial_t u = |D|u - |u|^2u, \qquad (t,x)\in \bbR\times\bbR,\\
        u(0,x)=u_0(x),
    \end{cases}
\end{align}
where $|D|$ is the Fourier multiplier operator with symbol $|\xi|$. This equation arises in a variety of physical contexts, including continuum limits of lattice models \cite{KLS2013CMP}, wave turbulence models \cite{Physics1997JNS,Physics2001PhysD}, and gravitational collapse \cite{Physics2007CPAM,FrohlichLenzmann2007CPAM}.

The half-wave equation \eqref{eq:half-wave} is a mass-critical dispersive model whose blow-up theory remains largely incomplete. One source of difficulty is that, despite its critical scaling, the equation lacks several classical symmetries, including Lorentz, Galilean, and pseudo-conformal invariance. Together with its nonlocal nature, this prevents a direct application of standard NLS techniques to the study of blow-up dynamics. The first finite-time blow-up result for negative-energy solutions was recently established by the third author in \cite{Park2026arXiv}. The aim of the present work is to further investigate the blow-up dynamics and to prove an upper bound of log-log type for the blow-up rate. This upper bound is conjectured to be sharp and is analogous to the classical log-log blow-up law in the mass-critical NLS.

Equation \eqref{eq:half-wave} is invariant under the scaling, space-time
translation, and phase rotation symmetries
\begin{equation*}
    \begin{aligned}
        u(t,x)&\mapsto \lambda^{\frac12}u(\lambda t,\lambda x),\\
        u(t,x)&\mapsto u(t+t_0,x+x_0),\\
        u(t,x)&\mapsto e^{i\gamma}u(t,x),
    \end{aligned}
    \qquad
    \begin{aligned}
        \lambda&>0,\\
        (t_0,x_0)&\in\bbR\times\bbR,\\
        \gamma&\in\bbR.
    \end{aligned}
\end{equation*}
The corresponding conserved quantities are the mass, momentum, and energy:
\begin{equation}\label{eq:conservation laws}
    \begin{gathered}
        M(u)=\int_{\bbR}|u(t,x)|^2\,dx,
        \qquad
        P(u)=\int_{\bbR} i\overline{u}(t,x)\,\partial_xu(t,x)\,dx,
        \\
        E(u)=\frac12\int_{\bbR}||D|^{1/2}u(t,x)|^2\,dx
        -\frac14\int_{\bbR}|u(t,x)|^4\,dx.
    \end{gathered}
\end{equation}
We refer to \eqref{eq:half-wave} as mass-critical (or $L^2$-critical) since
its scaling symmetry preserves the mass.

The ground state $Q$ is the positive optimizer for the sharp Gagliardo--Nirenberg inequality and solves
\begin{equation*}
    |D|Q+Q-Q^3=0.
\end{equation*}
It plays a central role in the analysis of \eqref{eq:half-wave}, as $e^{it}Q$ is a standing-wave solution. Frank and Lenzmann \cite{FrankLenzmann2013Acta} proved the existence and uniqueness of such $Q$ in $H^{1/2}(\bbR)$ up to translation.  Consequently, solutions below the mass $M(Q)$ are global, and $M(Q)$ is the natural threshold mass for the $L^2$-critical dynamics. 

In contrast to NLS, the half-wave equation admits a one-parameter family of traveling waves of the form $e^{it}Q_v(x-vt)$, with $|v|<1$, where the profile $Q_v$ solves
\begin{equation*}
    |D|Q_v+Q_v-|Q_v|^2Q_v+i v \partial_x Q_v=0.
\end{equation*}
Moreover, $Q_v$ recovers the ground state in the limit $v\to 0$, while $M(Q_v)\to 0$ as $|v|\to 1$ \cite{KLR2013ARMAhalfwave}.


We next recall some previous results on \eqref{eq:half-wave}. 
At the threshold mass $M(u_0)=M(Q)$, Krieger, Lenzmann, and Rapha\"el \cite{KLR2013ARMAhalfwave} constructed finite-time blow-up solutions in one dimension, with blow-up rate $\||D|^{1/2}u(t)\|_{L^2}\sim |t|^{-1}$ as $t\to0$; see also \cite{Georgiev2Dhalfwaveblowup,Georgiev3Dhalfwaveblowup} for related higher-dimensional constructions. 
Below the threshold, G\'erard, Lenzmann, Pocovnicu, and Rapha\"el \cite{GLPR2018annPDE} constructed a two-soliton solution displaying a transient turbulent regime, where large growth of higher Sobolev norms is followed by long-time saturation. 
Above the threshold, the current authors \cite{KimKwonPark2025arXiv} constructed Bourgain--Wang type blow-up solutions with the same rate as the minimal blow-up solution, and proved their instability.
In a companion paper \cite{Park2026arXiv}, the third author obtained the first finite-time blow-up result for negative-energy solutions, together with an upper bound on the blow-up rate. For further related works on the half-wave equation, we refer to \cite{Li2024WeakStability,CaoSuZhang2022,Li2022Inhomogeneous,BellazziniGeorgievVisciglia2018,ChoffrutPocovnicu2018}; for the half-wave maps equation, see \cite{GL2018LMP,GL2025sigma,GL2026arXiv} and references therein.

We now state our main theorem for even negative-energy solutions to \eqref{eq:half-wave}.
\begin{theorem}\label{thm:main thm}
    There exist constants $\alpha^*>0$ and $C^*>0$ with the following property.
    Let $u_0\in H^{1/2}_{\mathrm{e}}(\mathbb{R})$ satisfying 
    \begin{equation}\label{eq:main-assumptions}
        \begin{gathered}
            0<\alpha_0=\alpha(u_0)\coloneqq M(u_0)-M(Q)<\alpha^*, \qquad E_0\coloneqq E(u_0)<0.
        \end{gathered}
    \end{equation}
    Then, the corresponding solution $u(t)$ to \eqref{eq:half-wave} blows up in a finite time $T\in(0,\infty)$, and for all $0<t<T$, one has
    \begin{equation}\label{eq:main-upper-rate}
        \|u(t)\|_{\dot H^{1/2}}
        \le C^*\left(\frac{\log|\log(T-t)|}{T-t}\right)^{1/2}.
    \end{equation}
\end{theorem}

The upper bound in Theorem~\ref{thm:main thm} has the same log-log law as in the mass-critical NLS. This is somewhat unexpected for the half-wave equation, because its ground state has only slow algebraic decay, whereas the NLS ground state decays exponentially. The present theorem refines the finite-time blow-up result in \cite{Park2026arXiv} in the negative-energy regime by proving the corresponding log-log upper bound.


\begin{remark}[Method and novelties]
    We use the modulation method developed by Merle and Rapha\"el for the mass-critical NLS \cite{MerleRaphael2005AnnMath,MerleRaphael2003GAFA,Raphael2005MathAnnalen,MerleRaphael2004Invent,MerleRaphael2006JAMS,MerleRaphael2005CMP}.
    For the half-wave equation, the lack of pseudo-conformal symmetry and the nonlocality of $|D|$ prevent a direct adaptation of the NLS argument for log-log blow-up.
    In \cite{Park2026arXiv}, the first construction of blow-up dynamics of negative energy solutions for the half-wave equation was obtained. A central ingredient of that work was a local-virial spectral property, established by a computer-assisted proof. In the present paper, we use this spectral property as a key input. The main novelty here is instead the construction of an approximate blow-up profile adapted to the sharp upper bound on the log-log blow-up rate.
    We start from the tail-computation method, whose representative starting points include \cite{RaphaelRodnianski2012,MerleRaphaelRodnianski2013Invention,MerleRaphaelRodnianski2015CambJMath}.
    In the half-wave and fractional NLS setting, this method was used in \cite{KLR2013ARMAhalfwave,Lan2022IMRN,KimKwonPark2025arXiv}. We further develop this by summing the corrector profiles through \emph{Borel's integral summation method} on a family of \emph{$\Lambda$-analytic spaces}, denoted by $X_\rho$ below.

    We emphasize that the method is designed to be systematic. Once the corresponding local-virial spectral property is available, the same strategy should be adaptable to other mass-critical fractional Schr\"odinger type equations, including the original mass-critical NLS. Thus, our argument gives a route to the log-log upper bound that is not tied to pseudo-conformal symmetry. In particular, we expect that, in the fractional NLS setting considered in \cite{Lan2022IMRN}, this approach could strengthen the existing upper bound to the log-log upper bound. In Appendix~\ref{sec:comparison NLS}, we sketch how our profile construction is related to NLS works by Merle--Raphaël. 
\end{remark}

\begin{remark}[Comparison to \eqref{eq:NLS}]
    We recall results in the $L^2$-critical nonlinear Schr\"odinger equation
    \begin{equation}\label{eq:NLS}\tag{NLS}
        i\partial_t u+\Delta u+|u|^{\frac{4}{d}}u=0,\quad (t,x)\in I\times \bbR^d.
    \end{equation}
    At the threshold mass, blow-up uniqueness was proved in \cite{Merle1993Duke}; see also \cite{RS2011JAMSinhomo}. Above the threshold, pseudo-conformal type blow-up solutions exist \cite{BourgainWang1997}, but this dynamics is unstable \cite{MRS2013AJM}. The stable single-bubble regime is instead governed by the log-log law $\|u(t)\|_{\dot H^1}^{-1}\sim \sqrt{2\pi(T-t)/\log|\log(T-t)|}$. This regime was first constructed in \cite{Perelman2001} and then completely classified in the series of Merle and Rapha\"el \cite{MerleRaphael2005AnnMath,MerleRaphael2003GAFA,Raphael2005MathAnnalen,MerleRaphael2004Invent,MerleRaphael2006JAMS,MerleRaphael2005CMP}. More recently, beyond the single-bubble regime, the universal log-log bound was established by Kwak and the second author \cite{KwakKwon2026arXiv}. In particular, it ruled out the self-similar blow-up rate for general mass-critical NLS blow-up solutions.

    The comparison with \eqref{eq:half-wave} is useful, though not literal. As recalled above, threshold and Bourgain--Wang type blow-up constructions show some NLS picture persists for the half-wave flow, while the analogue of the NLS threshold uniqueness remains open. 
    However, the analogy breaks down at the level of the mechanisms available in the proof. The half-wave equation has no pseudo-conformal symmetry, and the Glassey virial argument does not turn negative energy into finite-time blow-up in this setting. At a technical level, the nonlocality of $|D|$ and the slow decay of solitary waves create further localization difficulties.
    Related negative-energy blow-up results for fractional NLS are known in \cite{BHL2016JFA,Lan2022IMRN,Dinh2019Nonlinearity}; however, they do not give finite-time blow-up for the half-wave equation, nor do they provide a sharp upper bound on the blow-up rate for fractional NLS.
    
    Thus, the NLS log-log theory is a motivation. It suggests the log-log scale as the natural candidate for stable critical blow-up, but the proof for \eqref{eq:half-wave} requires a method not tied to pseudo-conformal symmetry.
\end{remark}

\subsection{Strategy of the proof}

We outline the proof of Theorem~\ref{thm:main thm}. The overall scheme follows the modulation analysis of Merle and Rapha\"el for the mass-critical NLS \cite{MerleRaphael2003GAFA,MerleRaphael2004Invent,MerleRaphael2005AnnMath,MerleRaphael2005CMP,Raphael2005MathAnnalen,MerleRaphael2006JAMS}. The main new point is the construction of the profile. In the NLS case, the pseudo-conformal symmetry play a crucial role at the profile level. For the half-wave equation, this structure is absent, and the nonlocality of $|D|$ prevents a direct localization of the self-similar tail. We therefore construct the profile by a different, pseudo-conformal-free procedure.

Let $Q$ be the ground state. Formally freezing the scaling direction with parameter $b$ leads to the stationary self-similar equation
\begin{align*}
    |D|\td Q_b+\td Q_b-|\td Q_b|^2\td Q_b-ib\Lambda \td Q_b=0.
\end{align*}
An $L^2$-solution of this equation would correspond to an exact self-similar blow-up profile. The expected self-similar profile, however, has a non-$L^2$ tail, so this equation is used only as a reference equation. Our goal is to construct an $L^2$-family $\{Q_b\}$, with $Q_b\to Q$ as $b\to0$, whose error
\begin{align*}
    \Psi_b\coloneqq |D|Q_b+Q_b-|Q_b|^2Q_b-ib\Lambda Q_b
\end{align*}
is exponentially small. The profile is designed for the almost self-similar regime: the scaling direction is frozen at the profile level, and the desired leading modulation law is $b_s\approx0$.

The construction starts from a tail computation. We first construct real-valued correctors $R_j$ so that, for each fixed $N$, the modified profile satisfies
\begin{align*}
    Q_b^{(N)}\coloneqq Q+\sum_{j=1}^{N}(ib)^jR_j \quad \text{with} \quad \norm{\Psi_b^{(N)}}_{H^1}\lesssim_N |b|^{N+1},
\end{align*}
where $\Psi_b^{(N)}$ is the corresponding self-similar error. A key point is that this recursion can be continued to arbitrary order. See \eqref{eq:Psi cancel} for more details. 

However, the fixed-order expansion is not sufficient for the log-log analysis. The correctors are generated by repeated applications of the scaling vector field $\Lambda$, and they satisfy analytic bounds of the form
\begin{align*}
    \norm{R_k}_{X_{\rho}}\lesssim A^kk!.
\end{align*}
Here, $X_\rho$ is a $\Lambda$-analytic space in Section~\ref{sec:lambda analytic X}. This norm is tailored to the factorial growth of the profiles produced by the recursive tail computation, such as $\Lambda^k Q$. The truncation at order $N$ gives an error of size comparable to $N!|b|^{N+1}$. To obtain the exponential error bound $\|\Psi_b\|_{X\rho} \lesssim e^{-\frac{c}{|b|}}$ needed for the log-log law, we choose the truncation order $N$ depending on the dynamical parameter $b$; by Stirling's formula, the optimal choice is
\begin{equation*}
    N\sim |b|^{-1}.
\end{equation*}
Since such a $b$-dependent truncation would not give a regular family of profiles, we use Borel's integral summation method in $X_{\rho}$. See Section~\ref{sec:almost self similar}. This produces a genuine $C^1$ family $Q_b$ satisfying a desired error bound for $\Psi_b$.

With this profile in hand, we decompose negative energy solutions with even symmetry and mass close to the ground-state mass as
\begin{equation*}
    u(t,x)=\frac{e^{i\gamma(t)}}{\lambda(t)^{\frac12}}(Q_{b(t)}+\epsilon)\bigg(\frac{x}{\lambda(t)}\bigg),\quad 
    (\epsilon,\Lambda Q_b)_r = (\epsilon, i\Lambda Q_b)_r = (\epsilon, i\Lambda^2 Q_b)_r = 0.
\end{equation*}
The orthogonality conditions fix the parameters $\lambda$, $b$, and $\gamma$, and yield the first modulation law
\begin{align*}
    \frac{ds}{dt}=\frac1{\lambda(t)},\qquad
    \left|\frac{\lambda_s}{\lambda}+b\right|+|\gamma_s-1|+|b_s|\lesssim \norm{\epsilon}_{\calH^{1/2}}+e^{-\frac{c}{|b|}}.
\end{align*}
This estimate shows that the scaling is almost self-similar, but it does not yet give the sign or size of $b_s$ needed for the log-log upper bound.

The second modulation estimate comes from the local virial identity
\begin{align*}
    \partial_t\Phi(u)=2E(u_0),\qquad \Phi(u)=(iu,\Lambda u)_r.
\end{align*}
After inserting the sharp decomposition and using the orthogonality conditions, one obtains the identity
\begin{align}\label{eq:virial expansion}
    \partial_s\Phi(\epsilon)=(2e_1+O(b))b_s-2\lambda|E_0|,\quad e_1>0,\quad E_0=E(u_0)<0.
\end{align}
Thus the control of $b_s$ is reduced to a lower bound for the radiation contribution. At the linearized level, this contribution is controlled by the local-virial form
\begin{align*}
    \partial_s\Phi(\epsilon)\approx \mathbf H(\epsilon)\coloneqq \int_\bbR |\epsilon|^2 yQQ_y dy+2\int_\bbR \Re(\epsilon)^2 yQQ_y dy.
\end{align*}
The required coercivity is the following spectral property, proved in the companion work \cite{Park2026arXiv} by a computer-assisted proof. 

We use the weighted Sobolev norm
\begin{equation*}
    \norm{f}_{\mathcal H^{1/2}}\coloneqq
    \norm{f}_{\dot H^{1/2}}+\norm{\langle y\rangle^{-2}f}_{L^2}.
\end{equation*}
The weighted $L^2$ component reflects the profile weight $yQQ_y\sim\langle y\rangle^{-4}$ appearing in the local virial functional, and therefore matches its expected coercivity.

\begin{proposition}[Spectral property for local virial functional, \cite{Park2026arXiv}]\label{prop:Spectral structure of the bilinear form H}
    There exist universal constants $\delta, C>0$ such that for all $\epsilon\in H^{1/2}_{\mathrm{e}}$,
    \begin{equation}\label{eq:positivity of H}
        \mathbf H(\epsilon) \geq \delta \|\epsilon\|_{\calH^{1/2}}^2-C((\epsilon,Q)_r^2+(\epsilon,\Lambda Q)_r^2+(\epsilon, i\Lambda Q)_r^2+ (\epsilon,i\Lambda^2 Q)_r^2).
    \end{equation}
\end{proposition}
The orthogonality conditions and conservation laws eliminate the bad directions in this coercivity estimate. Combining \eqref{eq:virial expansion} with \eqref{eq:positivity of H} gives
\begin{align*}
    b_s\gtrsim \lambda|E_0|+\norm{\epsilon}_{\calH^{1/2}}^2-e^{-\frac{c}{|b|}}.
\end{align*}
Together with the first modulation law and the exponentially small profile error, this differential inequality yields the log-log upper bound. The only model-dependent input in this last step is the local-virial spectral property; once this input is available, the argument is not tied to pseudo-conformal symmetry and applies to the half-wave equation in the same framework that also recovers the classical NLS upper bound.

\begin{remark}[Outside even symmetry]
   To approach the general solution case, the spectral property in Proposition~\ref{prop:Spectral structure of the bilinear form H} has to be proved for general solutions. This may be achieved by extending the computer-assisted proof developed in \cite{Park2026arXiv}. 

   In the modulation analysis, one also needs to introduce the translation parameter $x(t)$ and its associated boost parameter $\nu$, as in \cite{KLR2013ARMAhalfwave,Lan2022IMRN,KimKwonPark2025arXiv}. The desired approximate profile should then take the form
    \begin{equation*}
        Q+\sum_{(j,k)\in \mathbf{N}} (ib)^j (i\nu)^k R_{j,k},\quad \text{for a suitable index set } \mathbf N.
    \end{equation*}
    with a corresponding profile error
    \begin{align*}
        \Psi_{b,\nu}(f)\coloneqq |D|f+f-|f|^2f-ib\Lambda f +i\nu \partial_x f+ ib X_{\nu}\partial_{\nu}f.
    \end{align*}
    Here, the modulation vector field is taken to be $X=X_\nu \partial_\nu$, since the desired regime has $b_s\approx 0$, or equivalently $X_b\approx 0$. We write
    \begin{align*}
        X_\nu=c_1 \nu+c_2 \nu^2+ c_3 \nu^3 +\cdots,\quad\text{for some}\quad c_k \in \bbR.
    \end{align*}
    Choosing the coefficients in $X_\nu$ inductively to satisfy the solvability conditions as far as possible, one obtains
    \begin{equation*}
        c_1=1,\qquad c_{2\ell}=0\quad(\ell\geq1),
    \end{equation*}
    and suitable real coefficients $c_{2\ell+1}$, $\ell\geq1$. With this choice, the solvability conditions hold precisely for the indices
    \begin{equation*}
         \mathbf{N}=(\bbN_{\geq 1}\times \{0\}) \cup (\bbN_{\geq 0}\times \{1\}) \cup \{(0,2)\}
    \end{equation*}
    With this choice, one expects to construct, by Borel summation, an approximate profile $Q_{b,\nu}$ whose error satisfies a bound of the form
    \begin{equation*}
        \|\Psi_{b,\nu}(Q_{b,\nu})\|_{Y} \lesssim e^{-c/|b|}+(|\nu|+|b|)|\nu|^2 .
    \end{equation*}
    This level of accuracy appears to be sufficient for proving the log-log upper bound, provided the required spectral property is available. Indeed, the spectral property is expected to produce a coercive contribution of size $\nu^2$, which precisely compensates for the additional error term $(|\nu|+|b|)|\nu|^2$ arising from the boost modulation. Consequently, following the Merle--Rapha\"el strategy, this suggests that the log-log upper bound should extend beyond the even-symmetric setting.
\end{remark}

\vspace{5bp}
\noindent\textbf{Acknowledgements.} This work was carried out in parallel with the companion paper by the third author \cite{Park2026arXiv}, which treats the spectral problem.
T.~Kim was supported by a KIAS Individual Grant (MG105201) at Korea Institute for Advanced Study. S.~Kwon was partially supported by the National Research Foundation of Korea, RS-2019-NR040050 and RS-2022-NR069873. J.~Park was partially supported by the National Research Foundation of Korea, RS-2019-NR040050, RS-2022-NR069873, and RS-2024-00333393.

\section{Notation and preliminaries}
\underline{\textit{Notations.}} Throughout the paper, we use the standard notation
$\mathbb{R}$, $\mathbb{C}$, and $\mathbb{N}$ for the sets of real numbers, complex numbers, and natural numbers, respectively. 
For a complex number $A = A_1 + iA_2 \in \mathbb{C}$ with $A_1, A_2 \in \mathbb{R}$, we denote $\Re A \coloneqq A_1$ and $\Im A \coloneqq A_2$.

For quantities $A\in\bbC$ and $B\geq0$, we denote $A \lesssim B$ or $A=\mathcal{O}(B)$ if $|A|\leq CB$ holds for some implicit constant $C$. For $A,B\geq0$, we say $A \sim B$ when $A \lesssim B$ and $B \lesssim A$. Similarly, for $A\geq 0$ and $B\in \bbC$, we write $A\gtrsim B$ if $B\lesssim A$. If $C$ depends on some parameters $m$, then we write $\lesssim_m,\sim_m$, and $\gtrsim_m$ to indicate this dependence. 

The Fourier transform (on $\bbR$) is denoted by 
\begin{align*}
    \mathcal{F}(f)(\xi)=\widehat{f}(\xi)\coloneqq\int_{\mathbb{R}}f(x)e^{-ix\xi}dx.
\end{align*}
The inverse Fourier transform is given by $\mathcal{F}^{-1}(f)(x)\coloneqq\frac{1}{2\pi}\int_{\R}\wt f(\xi)e^{ix\xi}d\xi$.
We denote $|D|$ by the Fourier multiplier with symbol $|\xi|$, that is, $|D|\coloneqq\mathcal{F}^{-1}|\xi|\mathcal{F}$.

Denote by $L^{p}(\mathbb{R})$ and $H^{s}(\bbR)$ the standard $L^{p}$
and Sobolev spaces on $\bbR$. We also denote the weighted norm 
\begin{equation*}
    \norm{f}_{\langle y\rangle^{-2}L^\infty(\bbR)}\coloneqq \norm{\langle y\rangle^{2}f}_{L^\infty(\bbR)}.    
\end{equation*}
As we work on $\R$, we often omit
$\R$ when there is no confusion. 

We also use the real inner product defined by 
\begin{align*}
    (f,g)_{r}\coloneqq\Re\int_{\bbR}f\ol{g}dx.
\end{align*}
We denote by $\Lambda$ the $L^2$-scaling generator in
$\bbR$ as 
\begin{align}\label{eq:Lambda oper}
	\Lambda f  \coloneqq\frac{d}{d\lambda}\bigg|_{\lambda=1}\lambda^{\frac{1}{2}}f(\lambda\cdot)=\left(\frac{1}{2}+x\partial_{x}\right)f.
\end{align}

\underline{\textit{Preliminary lemmas.}}
We collect a few preliminary lemmas.
\begin{lemma}[Fractional weighted Hardy inequality]
    For $f \in \calH^{1/2}$, we have
    \begin{equation}\label{eq:Hardy inequ}
        \|\langle y\rangle^{-1}f\|_{L^2}\lesssim \|f\|_{\calH^{1/2}}.
    \end{equation}
\end{lemma}
\begin{proof}
    We first recall $\dot H^{1/2}(\mathbb R)\hookrightarrow \mathrm{BMO}(\mathbb R)$; see, for instance, \cite[Chapter~VI]{Steinharmonic-book}. Thus, it suffices to prove
    \begin{align*}
        \norm{\langle y\rangle^{-1}f}_{L^2}
        \lesssim
        \norm{f}_{\mathrm{BMO}}
        +
        \norm{\langle y\rangle^{-2}f}_{L^2}.
    \end{align*}
    Let $I_k=(-2^k,2^k)$ and $m_k\coloneqq \frac{1}{|I_k|}\int_{I_k} f(x) dx$. Since $\langle y\rangle^{-2}\simeq 1$ on $I_0$, we have $|m_0|\lesssim \norm{\langle y\rangle^{-2}f}_{L^2}$. Also, for $k\ge1$,
    \begin{align*}
        |m_k-m_{k-1}|
        \le
        |I_{k-1}|^{-1}{\textstyle\int_{I_{k-1}}}|f-m_k|dy
        \lesssim
        \norm{f}_{\mathrm{BMO}},
    \end{align*}
    and hence $|m_k|\lesssim \norm{\langle y\rangle^{-2}f}_{L^2}+k\norm{f}_{\mathrm{BMO}}$. By John--Nirenberg,
    \begin{align*}
        {\textstyle\int_{I_k}}|f-m_k|^2dy
        \lesssim
        |I_k|\norm{f}_{\mathrm{BMO}}^2
        \lesssim
        2^k\norm{f}_{\mathrm{BMO}}^2.
    \end{align*}
    Now set $A_0=I_0$ and $A_k=\{2^{k-1}\le |y|<2^k\}$ for $k\ge1$. Then
    \begin{align*}
        {\textstyle\int_{A_0}}\langle y\rangle^{-2}|f|^2dy
        &\lesssim \norm{\langle y\rangle^{-2}f}_{L^2}^2,
        \\
        {\textstyle\int_{A_k}}\langle y\rangle^{-2}|f|^2dy
        &\lesssim
        2^{-2k}{\textstyle\int_{I_k}}|f-m_k|^2dy
        +
        2^{-2k}|I_k||m_k|^2\\
        &\lesssim
        2^{-k}\norm{f}_{\mathrm{BMO}}^2
        +
        2^{-k}(\norm{\langle y\rangle^{-2}f}_{L^2}+k\norm{f}_{\mathrm{BMO}})^2.
    \end{align*}
    Summing over $k$ gives
    \begin{align*}
        \norm{\langle y\rangle^{-1}f}_{L^2}^2
        &\lesssim
        \norm{\langle y\rangle^{-2}f}_{L^2}^2
        +
        \sum_{k\ge1}2^{-k}\{\norm{\langle y\rangle^{-2}f}_{L^2}^2+(\norm{\langle y\rangle^{-2}f}_{L^2}+k\norm{f}_{\mathrm{BMO}})^2\}\\
        &\lesssim
        \norm{\langle y\rangle^{-2}f}_{L^2}^2+\norm{f}_{\mathrm{BMO}}^2. \qedhere
    \end{align*}
\end{proof}

\begin{lemma}[Stirling formula] For $j\in \N$, we have
    \begin{equation}\label{eq:Stirling}
        j!= \sqrt{2\pi j}(\tfrac je)^j(1+\tfrac{1}{12j}+O(\tfrac{1}{j^2})).
    \end{equation}
\end{lemma}

\begin{lemma}[Poisson left-tail bound]\label{lem:poisson_left_tail}
    Let $\delta\in(0,1)$ and $\lambda>0$. Then for $j\in \N$ with $0\le j\le (1-\delta)\lambda$, we have
    \begin{equation}\label{eq:poisson_tail_bound2}
        e^{-\lambda}\sum_{m=0}^{j}\frac{\lambda^m}{m!}
        \le \exp \bigl(-c(\delta)\lambda\bigr),
    \end{equation}
    where $c(\delta)\coloneqq \delta+(1-\delta)\log  (1-\delta)>0$.
\end{lemma}
\begin{proof}
    Let $X\sim\mathrm{Poisson}(\lambda)$. For $t>0$, Markov's inequality gives
    \[
        e^{-\lambda}\sum_{m=0}^{j}\frac{\lambda^m}{m!}
        = \mathbb P(X\le j)
        = \mathbb P(e^{-tX}\ge e^{-tj})
        \le e^{tj}\mathbb E(e^{-tX})
        = \exp(tj+\lambda(e^{-t}-1)).
    \]
    Choosing $t=-\log(1-\delta)$ and using $j\le (1-\delta)\lambda$, we obtain
    \[
        tj+\lambda(e^{-t}-1)
        \le -\lambda\{\delta+(1-\delta)\log(1-\delta)\}
        = -c(\delta)\lambda.
    \]
    This proves \eqref{eq:poisson_tail_bound2}.
\end{proof}

\underline{\textit{Linearized operator.}}
One can linearize \eqref{eq:half-wave} at a complex-valued function $v$. we denote the linearized operator $L_v$, that is $\R$-linear: for $f\in L^2(\R)$
\[
    L_v f \coloneqq |D|f + f - |v|^2 f - 2\Re\{\overline{v} f\}v.
\]
In particular, for $v=Q$,
the linearized operator $L_Q$ satisfies the following kernel relations \cite{KLR2013ARMAhalfwave}:
\begin{equation*}
    \begin{aligned}
        L_Q[iQ]&=0, \qquad L_Q[\nabla Q]=0.
    \end{aligned}
\end{equation*}
We also recall the following coercivity property of $L_Q$.

\begin{lemma}[Coercivity, \cite{KLR2013ARMAhalfwave}]\label{lem:estimate of LQ inverse}
    Let $\sigma\ge 1$. The following statements hold.
    \begin{enumerate}
        \item (Solvability condition) There exists the unique solution $f \in H^{\sigma+1}$ to
        \begin{equation}\label{eq: LQf g}
            L_Qf=g
        \end{equation}
        if and only if $g\in H^\sigma$ satisfies the solvabiility conidtion,
        \begin{equation}
            (g,\nabla Q)_r=(g,iQ)_r=0. \label{eq:solvable}
        \end{equation}
        In this case, we denote $f\coloneqq L_Q^{-1}[g]$.

        \item Let $g\in H^\sigma$ satisfy \eqref{eq:solvable}. Then, we have
        \begin{equation}
            \norm{L_Q^{-1}[g]}_{H^{\sigma+1}} \lesssim \norm{g}_{H^\sigma}. \label{eq:LQ inverse esti}
        \end{equation}
        Moreover,
        \begin{equation}
            \norm{\langle y \rangle^2L_Q^{-1}[g]}_{L^\infty} \lesssim \norm{\langle y \rangle^{2} g}_{L^\infty}. \label{eq:L_Q inverse decay estimate}
        \end{equation}
    \end{enumerate}
\end{lemma}

\section{$\Lambda$-analytic space $X_\rho$}\label{sec:lambda analytic X}
In this section, we introduce the \emph{$\Lambda$-analytic spaces} where $\Lambda$ is the scaling generator \eqref{eq:Lambda oper}. We will use it in the construction of the blow-up profile in Section~\ref{sec:profile}. The present section is devoted to their definitions and basic properties. We formulate the definitions for an abstract Banach algebra $Y$, so that the same notation applies to the several underlying spaces used later.

Let $Y$ be a Banach algebra of complex-valued functions on $\mathbb R$ under pointwise multiplication; that is, there exists a constant $C_Y>0$ such that $\norm{fg}_{Y}\le C_Y\norm{f}_{Y}\norm{g}_{Y}$ for all $f,g\in Y$.
We define the \(\Lambda\)-analytic space \(X_\rho^{(Y)}\), equipped with the norm
\begin{align*}
    \norm{f}_{X_\rho^{(Y)}}
    \coloneqq
    \sum_{m\ge 0}\frac{\rho^m}{m!}\norm{\Lambda^m f}_{Y},
    \qquad \rho >0.
\end{align*}
Then, we have the following lemma.
\begin{lemma}[$\Lambda$-analytic space $X_\rho^{(Y)}$]\label{lem:Property of X_rho}
    Let $Y$ be a Banach algebra. Then, $X_\rho^{(Y)}$ is a Banach space with the following properties:
    \begin{enumerate}
        \item If $0<\rho<\rho'$, then
        \begin{equation}\label{eq:inclusion for rho}
            \norm{f}_{X_\rho^{(Y)}} \leq \norm{f}_{X_{\rho'}^{(Y)}}.
        \end{equation}

        \item (Banach algebra) For each $\rho>0$ and for all $f,g\in X_\rho^{(Y)}$, we have
        \begin{equation}\label{eq: X_rho is Banach algebra}
            \norm{fg}_{X_\rho^{(Y)}} \lesssim_{\rho, Y}
            \norm{f}_{X_\rho^{(Y)}}\norm{g}_{X_\rho^{(Y)}}.
        \end{equation}

        \item ($\Lambda$-shift) Let $\delta \in (0,1]$. For $f\in X_{\rho+\delta}^{(Y)}$, we have
        \begin{equation}\label{eq:shift_Cauchy_delta3}
            \norm{\Lambda f}_{X_\rho^{(Y)}}\leq \delta^{-1}\norm{f}_{X_{\rho+\delta}^{(Y)}}.
        \end{equation}
        In particular, for all $\rho>0$, we have
        \begin{equation}\label{eq:derivative loss of X_rho} 
            \norm{|D|f}_{X_\rho^{(H^1)}} \lesssim e^\rho \norm{f}_{X_\rho^{(H^2)}}
        \end{equation}
    \end{enumerate}
    Consider the case 
    \begin{equation*}
        Y\in \{H^1, H^2, \langle y\rangle^{-2}L^\infty\}.
    \end{equation*}
    Then there exists a constant $\rho^*>0$ such that the following properties hold for every $\rho\in(0,\rho^*]$.
    \begin{enumerate}\setcounter{enumi}{3}
        \item Let $Q$ be the ground state for \eqref{eq:half-wave}. Then
        \begin{equation}\label{eq:Q in X}
            Q\in X_\rho^{(Y)},\qquad \Lambda Q\in X_\rho^{(Y)}.
        \end{equation}

        \item (Coercivity estimate on $X_\rho$) Let $g \in X_\rho^{(Y)}$ satisfy the solvability condition \eqref{eq:solvable}. Then
        \begin{equation}\label{eq:LQ invert}
            \norm{L_Q^{-1}[g]}_{X_\rho^{(Y)}} \lesssim_Y \norm{f}_{X_\rho^{(Y)}}.
        \end{equation}
    \end{enumerate}
\end{lemma}
\begin{proof}
    For a notational convenience, we simply denote $X_\rho=X_\rho^{Y}$, whenever no confusion arises.
    First, by definition, $X_\rho$ is a Banach space. Moreover, \eqref{eq:inclusion for rho} follows immediately from the monotonicity in $\rho$.

    \vspace{5bp}

    \noindent \textbf{Step 1.} We prove \eqref{eq: X_rho is Banach algebra}. Denote $Z=x\partial_x$. Thus, $Z=\Lambda-\frac12$. Since $Z$ satisfies the Leibniz rule, we have
    \begin{equation}\label{eq:Lambda_m_product_multinomial}
        \Lambda^m(fg)=\sum_{\substack{a,b,r\geq0\\ a+b+r=m}}\frac{m!}{a!\,b!\,r!}2^{-r}(Z^a f)(Z^b g).
    \end{equation}
    Again using $Z=\Lambda-\frac12$, we have
    \begin{equation}
        \norm{Z^a f}_{Y}\leq \sum_{p=0}^a\binom{a}{p}2^{-(a-p)}\norm{\Lambda^p f}_{Y}. \label{eq: Zaf binom esti}
    \end{equation}
    Thanks to \eqref{eq:Lambda_m_product_multinomial} and trinomial expansion, we derive 
    \[
        \norm{\Lambda^m(fg)}_{Y} \lesssim_{Y} \sum_{\substack{p,q\ge0\\ p+q\le m}}
        \frac{m!}{p!\,q!\,(m-p-q)!}\Big(\frac32\Big)^{m-p-q}
        \norm{\Lambda^p f}_{Y}\norm{\Lambda^q g}_{Y}.
    \]
    Multiplying by $\rho^m/(m!)$ and summing in $m$, we obtain
    \[
        \norm{fg}_{X_\rho} \lesssim_Y \bigg(\sum_{m\ge0}\frac{1}{m!}\bigg(\frac32\rho\bigg)^m\bigg) \norm{f}_{X_\rho}\norm{g}_{X_\rho},
    \]
    which is exactly \eqref{eq: X_rho is Banach algebra}.

    \vspace{5bp}

    \noindent \textbf{Step 2.} In this step, we show \eqref{eq:shift_Cauchy_delta3} and \eqref{eq:derivative loss of X_rho}. We first prove \eqref{eq:shift_Cauchy_delta3}. By definition,
    \[
        \norm{\Lambda f}_{X_\rho} =\sum_{m\ge0}\frac{\rho^m}{m!}\norm{\Lambda^{m+1}f}_{Y} =\sum_{k\ge1}\frac{\rho^{k-1}}{(k-1)!}\norm{\Lambda^{k}f}_{Y}.
    \]
    Since 
    \begin{equation*}
        (\rho+\delta)^k-\rho^k=\int_\rho^{\rho+\delta}k s^{k-1}\,ds\ge \delta k\rho^{k-1},
    \end{equation*}
    we have $\rho^{k-1}\le (\delta k)^{-1}(\rho+\delta)^k$ for all $k\ge1$. Hence, we obtain
    \begin{align*}
        \norm{\Lambda f}_{X_\rho} \leq \delta^{-1}\sum_{k\ge1}\frac{(\rho+\delta)^k}{k!}\norm{\Lambda^{k}f}_{Y} \leq \delta^{-1}\norm{f}_{X_{\rho+\delta}}.
    \end{align*}

    We now prove \eqref{eq:derivative loss of X_rho}. From the commutator formula $[|D|,\Lambda] = |D|$, we have
    \[
        \Lambda^m (|D|f) = \sum_{k=0}^m\binom{m}{k}(-1)^k |D|(\Lambda^{m-k}f). 
    \]
    Therefore, by setting $m=n-k$ we obtain
    \begin{align*}
        \norm{|D|f}_{X_\rho^{(H^1)}} &\lesssim \sum_{m\geq 0}\frac{\rho^m}{m!}\sum_{k=0}^m \binom{m}{k}\norm{\Lambda^{m-k}f}_{H^2}
        \\
        &\leq \sum_{n\geq 0} \frac{\rho^n}{n!}\norm{\Lambda^n f}_{H^2}\bigg(\sum_{k\geq 0} \frac{\rho^k}{k!}\bigg) \leq e^\rho \norm{f}_{X_\rho^{(H^2)}}.
    \end{align*}

    \vspace{5bp}
    
    \noindent \textbf{Step 3.} From this step, we suppose $Y\in \{H^1,H^2,\langle y\rangle^{-2}L^\infty\}$. We prove \eqref{eq:Q in X} through the scaling orbit of $Q$. Set
    \begin{equation*}
        Q_\mu(x)\coloneqq e^{\mu/2}Q(e^\mu x).
    \end{equation*}
    Then $Q_\mu$ solves
    \begin{equation}\label{eq:scaled_Qmu}
        |D|Q_\mu+e^\mu Q_\mu-Q_\mu^3=0.
    \end{equation}
    Let $Y_{\mathrm{e,r}}$ be the closed subspace of real-valued even functions in $Y$, viewed as a real Banach space. For $h\in Y_{\mathrm{e,r}}$, the right-hand side below is real-valued and even, hence satisfies \eqref{eq:solvable}. Lemma~\ref{lem:estimate of LQ inverse} therefore defines a bounded map on this subspace, and \eqref{eq:scaled_Qmu} with $Q_\mu=Q+h$ is equivalent to
    \begin{equation*}
        h=L_Q^{-1}\bigl[(1-e^\mu)(Q+h)+3Qh^2+h^3\bigr].
    \end{equation*}
    Define
    \begin{equation*}
        F(\mu,h)\coloneqq h-L_Q^{-1}\bigl[(1-e^\mu)(Q+h)+3Qh^2+h^3\bigr].
    \end{equation*}
    We note that $F$ is analytic: $e^\mu$ has a scalar power series, the $h$-dependence consists of continuous multilinear products on the Banach algebra $Y$, and $L_Q^{-1}$ is bounded. Moreover, $F(0,0)=0$ and $D_hF(0,0)=I$. By the real-analytic implicit function theorem in Banach spaces, there exist $\mu_0>0$, a neighborhood $\mathcal U$ of $0$ in $Y_{\mathrm{e,r}}$, and a unique real-analytic map $h:(-\mu_0,\mu_0)\to\mathcal U$ such that $h(0)=0$ and $F(\mu,h(\mu))=0$ for $|\mu|<\mu_0$; the uniqueness is among solutions in $\mathcal U$. By the smoothness and decay of $Q$, $Q_\mu-Q\to0$ in $Y$ as $\mu\to0$; since it solves the same equation, uniqueness gives $h(\mu)=Q_\mu-Q$.
    
    Thus, for some $r>0$, $Q_\mu-Q=\sum_{m\ge1}a_m\mu^m$ in $Y$ and $\sum_{m\ge1}\norm{a_m}_{Y}\mu^m<\infty$. The distributional identity $\partial_\mu^m Q_\mu|_{\mu=0}=\Lambda^mQ$ gives $a_m=\Lambda^mQ/m!$, so
    \begin{equation*}
        \sum_{m\ge1}\frac{\mu^m}{m!}\norm{\Lambda^m Q}_{Y}<\infty.
    \end{equation*}
    In particular, after decreasing $\mu$ if necessary, there is $C_\Lambda>0$ such that
    \begin{equation} \label{eq: Lambdaj Q esti}
        \norm{\Lambda^m Q}_{Y}\leq C_\Lambda^{m+1} m!,
        \qquad m\ge0.
    \end{equation}
    Therefore, after reducing $\rho^*>0$ if necessary, $Q\in X_\rho^{(Y)}$ and $\Lambda Q\in X_\rho^{(Y)}$ for every $\rho\in(0,\rho^*]$.

    \vspace{5bp}

    \noindent \textbf{Step 4.} We prove \eqref{eq:LQ invert}. By \eqref{eq: X_rho is Banach algebra} and \eqref{eq:Q in X}, we have $Q^2\in X_\rho^{(Y)}$ for some $\rho>0$. From \eqref{eq: Zaf binom esti} with $Q^2\in X_\rho^{(Y)}$, we deduce
    \begin{equation}\label{eq: growth of Z of Q^2}
        \norm{Z^m (Q^2)}_{Y} \lesssim (C_*)^m m!,
    \end{equation}
    for some constant $C_*>0$. Now, we define a family of operator $\{\mathrm{ad}_{\Lambda}^{(k)}\}_{k\geq 0}$ by
    \[
        \mathrm{ad}_{\Lambda}^{(0)}\coloneqq L_Q,\quad \mathrm{ad}_{\Lambda}^{(k+1)} \coloneqq [\Lambda, \mathrm{ad}_{\Lambda}^{(k)}] \quad \text{ for } k\geq 1.
    \]
    Using $[|D|,\Lambda]=|D|$, for each $k \geq 1$ we obtain
    \begin{equation*}
        \mathrm{ad}_{\Lambda}^{(k)}h = (-1)^k|D|h - 2\Re\{Z^k(Q^2)h\} - Z^k(Q^2)h.
    \end{equation*}
    Therefore, we obtain 
    \[
        L_Q[\Lambda^m u] = \Lambda^m L_Q[u] - \sum_{k=1}^m\binom{m}{k} \mathrm{ad}_{\Lambda}^{(k)}[\Lambda^{m-k}u],
    \]
    and let by setting $u = L_Q^{-1}[f]$, where $f \in X_\rho^{(Y)}$ with $(f,\nabla Q)_r = (f,iQ)_r = 0$, we obtain
    \begin{equation}\label{eq:LQ Lambda LQ inverse}
        L_Q[\Lambda^m \{L_Q^{-1}[f]\}] = \Lambda^m f - \sum_{k=1}^m \binom{m}{k}\mathrm{ad}_{\Lambda}^{(k)}[\Lambda^{m-k}\{L_Q^{-1}[f]\}].
    \end{equation}
    Now, let $P$ be the $L^2$-orthogonal projection onto the complement of the kernel directions of $L_Q$ $\mathrm{Ran}(P)=\mathrm{span}\{\nabla Q,iQ\}^\perp$. Then $P$ is bounded on $Y$ and the restricted inverse
    \[
        T\coloneqq (L_Q|_{\mathrm{Ran}(P)})^{-1}P
    \]
    Then, we have
    \begin{equation}\label{eq:key_Tad_bound}
        \norm{T\circ\mathrm{ad}_\Lambda^{(k)}(L_Q)[h]}_{Y}\lesssim (C_*)^k k!\norm{h}_{Y}.
    \end{equation}
    Indeed, the $|D|h$ term is handled after applying $T$:
    since $|D|h=L_Q[h]-h+2\Re(Q^2h)+Q^2h$ and $T\circ L_Q= P$, we have
    \[
        T(|D|h)=Ph-Th+T\bigl(2\Re(Q^2h)+Q^2h\bigr),
    \]
    hence $\norm{T(|D|h)}_{Y}\lesssim \norm{h}_{Y}$. For the potential terms, using \eqref{eq: growth of Z of Q^2} and from Lemma~\ref{lem:estimate of LQ inverse}, we have
    \[
        \norm{T(Z^k(Q^2)h)}_{Y}\lesssim \norm{Z^k(Q^2)h}_{Y} \lesssim \norm{Z^k(Q^2)}_{Y}\norm{h}_{Y} \lesssim (C_Z)^k k!\norm{h}_{Y}.
    \]
    Now, we are ready to compute $\norm{L_Q^{-1}[f]}_{X_\rho}$. We decompose $\norm{L_Q^{-1}[f]}_{X_\rho^{(Y)}}$ into
    \begin{align*}
        \norm{L_Q^{-1}[f]}_{X_\rho} 
        \leq \sum_{m\geq0} \frac{\rho^m}{m!} \norm{P\Lambda^m L_Q^{-1}[f]}_{Y}
        +
        \sum_{m\geq0} \frac{\rho^m}{m!} \norm{(I-P)\Lambda^m L_Q^{-1}[f]}_{Y}.
    \end{align*}
    For the first term, applying $T$ both sides of \eqref{eq:LQ Lambda LQ inverse} and
    using Lemma~\ref{lem:estimate of LQ inverse} and the relation $P=T\circ L_Q$ combining with estimate \eqref{eq:key_Tad_bound}, we obtain
    \begin{align}\label{eq: P summation estimate}
        \sum_{m\geq0} \frac{\rho^m}{m!} \norm{P\Lambda^m L_Q^{-1}[f]}_{Y} \lesssim \norm{f}_{X_{\rho}}
        +\sum_{m\ge1}\frac{\rho^m}{m!}\sum_{k=1}^m (C_*)^k k!\norm{\Lambda^{m-k}L_Q^{-1}[f]}_{Y}.
    \end{align}
    Making the change of variable $n = m-k$, for $(C_Z+1)\rho < \tfrac12$, we have
    \begin{equation}\label{eq: P summation estimate 2}
        \begin{aligned}
            \sum_{m\geq 1} \frac{\rho^m}{m!} \sum_{k=1}^m (C_*)^kk!\norm{\Lambda^{m-k}L_Q^{-1}[f]}_{Y} 
            &= \sum_{n\geq 0}\sum_{k\geq 1} \frac{\rho^n}{n!}(C_Z\rho)^k\norm{\Lambda^n L_Q^{-1}[f]}_{Y}\\
            &\leq 2C_*\rho\norm{L_Q^{-1}[f]}_{X_\rho}.
        \end{aligned}
    \end{equation}
    For the second term, since $\mathrm{Ran}(I-P) = \{\nabla Q,iQ\}$ and $\Lambda^* = -\Lambda$, we estimate
    \begin{equation}\label{eq: I-P estimate}
        \begin{aligned}
            \norm{(I-P)\Lambda^m L_Q^{-1}[f]}_{H^1} &\lesssim |(\Lambda^m L_Q^{-1}[f],\nabla Q)_r| + |((\Lambda^m L_Q^{-1}[f],iQ)_r)|\\
            &\lesssim \norm{L_Q^{-1}[f]}_{L^2}(\norm{\Lambda^m (iQ)}_{L^2}+\norm{\Lambda^m\nabla Q}_{L^2})\\
            &\lesssim \norm{L_Q^{-1}[f]}_Y (\norm{\Lambda^m (iQ)}_{Y}+\norm{\Lambda^m\nabla Q}_{Y}).
        \end{aligned}
    \end{equation}
    For each $m \in \bbN$, we have
    \begin{equation}\label{eq: equation for Lambda^m nabla Q}
        \Lambda^m \nabla Q = \nabla \Bigl(\sum_{k=0}^m \binom{m}{k}(-1)^{m-k}\Lambda^k Q\Bigr). 
    \end{equation}
    From \eqref{eq: equation for Lambda^m nabla Q} and \eqref{eq: Lambdaj Q esti}, we estimate
    \[
        \norm{\Lambda^m \nabla Q}_{Y} \lesssim (C_\Lambda +1)^mm!.
    \]
    Using \eqref{eq:LQ inverse esti} and \eqref{eq: I-P estimate}, we obtain
    \begin{equation}\label{eq: I-P summation estimate}
        \sum_{m\geq 0}\frac{\rho^m}{m!} \norm{(I-P)\Lambda^m (L_Q)^{-1}[f]}_{Y} \leq \Bigl(\sum_{m\geq 0} ((C_\Lambda+1)\rho)^m \Bigr)\norm{f}_{Y}
    \end{equation}
    Combining \eqref{eq: P summation estimate}, \eqref{eq: P summation estimate 2} and \eqref{eq: I-P summation estimate}, we obtain
    \[
        \norm{L_Q^{-1}[f]}_{X_\rho} \leq C(\norm{f}_{X_\rho} + \rho\norm{L_Q^{-1}[f]}_{X_\rho}).
    \]
    Therefore, by choosing choosing sufficiently small $\rho^*>0$ so that $1-C\rho>0$, we obtain \eqref{eq:LQ invert}.
\end{proof}
Throughout the rest of the paper, unless otherwise specified, we take
\begin{equation}\label{eq:def Y}
    Y\in \{H^1, H^2, \langle y\rangle^{-2}L^\infty\}.
\end{equation}
For notational convenience, we often suppress the superscript $(Y)$ and write
\begin{equation}\label{eq:X simply denote}
    X_\rho = X_\rho^{(Y)},
    \qquad
    X=X^{(Y)} \coloneqq X_{\rho^*/2}^{(Y)},
\end{equation}
where \(\rho^*>0\) is the constant given in Lemma~\ref{lem:Property of X_rho}.
When the underlying space $Y$ needs to be emphasized, we restore the notation
$X_\rho^{(Y)}$ and $X^{(Y)}$.

\begin{remark}
    In this work, we handle solutions with initial data in $H^{1/2}$. The norms of $Y$ in the above are not relevant to the regularity of solutions. Instead, they are a part of the analytic framework used to construct and estimate the blow-up profile.
\end{remark}

\section{Construction of blow-up profile}\label{sec:profile}

In this section, we construct the almost self-similar profile for the modulation analysis. The starting point is the ground state $Q$, which solves
\begin{align*}
    |D|Q + Q - |Q|^2Q = 0.
\end{align*}
The variations of the scale parameter $\lambda$ act on the rescaled profile through the scaling generator $\Lambda$ and determined by the modulation law $\lambda_t+b\approx 0$. At a formal level, one may regard $b$ as the parameter measuring the self-similar scaling direction. Freezing the non-zero $b$ leads to the stationary equation for the corresponding self-similar profile:
\begin{align*}
    |D|\td Q_b + \td Q_b - |\td Q_b|^2\td Q_b - i b \Lambda \td Q_b = 0.
\end{align*}
If this equation admitted a nontrivial $L^2$ solution, it would generate a self-similar blow-up solution. However, the corresponding profile is expected to develop a non-integrable tail and therefore not belong to $L^2$. Consequently, the self-similar equation serves only as a reference equation: our goal is to construct an $L^2$ profile whose error with respect to this equation is sufficiently small.

This point of view already appears in the work of Merle--Rapha\"el on the mass-critical NLS \cite{MerleRaphael2003GAFA,MerleRaphael2004Invent,MerleRaphael2005AnnMath,MerleRaphael2005CMP,Raphael2005MathAnnalen,MerleRaphael2006JAMS}.
In that setting, the pseudo-conformal symmetry plays a decisive role: after a suitable reduction, the self-similar profile equation can be treated as a local equation for a real-valued profile. The resulting profile has an outgoing tail, but a spatial localization produce an $L^2$ profile, with only an exponentially small error. This localization mechanism is an essential part of the construction of the log-log profile.

In the present half-wave equation, this profile-level localization is no longer available. The operator $|D|$ is nonlocal, and cutting off the tail produces nonlocal errors which are not controlled by the Merle--Raphaël argument. Moreover, the pseudo-conformal symmetry underlying the mass-critical NLS construction has no direct counterpart for the half-wave equation.

We therefore develop a different construction based on tail computation and \emph{Borel's integral summation method} in a suitable \emph{$\Lambda$-analytic space}. We note that related tail computations for fractional NLS have appeared in \cite{KLR2013ARMAhalfwave,Lan2022IMRN,KimKwonPark2025arXiv}.
More precisely, we seek a family $Q_b$, with $Q_b \to Q$ as $b\to0$, and define its error by
\begin{align*}
    \Psi_b\coloneqq |D|Q_b + Q_b - |Q_b|^2Q_b - i b\Lambda Q_b.
\end{align*}
The goal is to construct $Q_b$ so that
\begin{align*}
    \norm{\Psi_b}_{X} \lesssim e^{-\frac{c}{|b|}},
\end{align*}
for a suitable $\Lambda$-analytic space $X$. 

The construction begins with a formal expansion around the ground state. The corrector  profiles are chosen recursively so that the self-similar profile error is cancelled order by order in powers of $b$. Since this expansion is generated by repeated applications of the scaling operator $\Lambda$, the correctors naturally exhibit factorial growth. Thus, iterating this procedure up to order $N$ gives a error of size $N!|b|^{N+1}$. To obtain an exponentially small error of size $e^{-\frac{c}{|b|}}$, the corresponding optimal truncation occurs at the effective order $N\sim |b|^{-1}$.

This optimal truncation is only heuristic, since the $b$-dependent integer cutoff would not produce a $C^1$ family of profiles. We instead define $Q_b$ by Borel summation. The $\Lambda$-analytic space $X$ is introduced precisely for this purpose: its norm captures the factorial growth of the correctors produced by the recursive tail computation, and the resulting Borel sum gives an actual $L^2$ profile with the required exponentially small error.

\subsection{Tail-computation}
We carry out the tail-computation for the formal correctors. This method was first used for the half-wave equation in \cite{KLR2013ARMAhalfwave}, where the blow-up profile was constructed through the solvability structure of the linearized operator around $Q$. In our construction, the same solvability structure allows us to determine the correction profiles recursively, so that the error in the self-similar equation is cancelled to arbitrary order.

\begin{proposition}[Tail computation]\label{prop: defining corrector profile infinitely}
    There exist a sequence of real-valued corrector profiles $\{R_j \in H^{1/2}: j\in \mathbb N \}$ such that for each $N\in \N$ a modified profile
    \begin{equation}\label{eq:Qb_N_def}
        Q_b^{(N)}\coloneqq Q+\sum_{j=1}^{N}(ib)^jR_j.
    \end{equation}
    satisfies
    \begin{equation}
        \norm{\Psi_b^{(N)}}_Y \lesssim_{N} |b|^{N+1}, \label{eq:psi N esti}
    \end{equation}
    where $\Psi_b^{(N)}$ is the profile error given by
    \begin{equation*}
        \Psi_b^{(N)} \coloneqq |D|Q_b^{(N)} + Q_b^{(N)} - |Q_b^{(N)}|^2Q_b^{(N)} - i b\Lambda Q_b^{(N)}.
    \end{equation*}
    Moreover, there exists a constant $A>0$ and $\rho^*>0$ such that, for any $0<\rho\le \frac{3\rho^*}{4}$ and $k\geq 1$, $R_k \in X_{\rho}$ with
    \begin{equation}\label{eq:Rk_gevrey_Xrho0}
        \norm{R_k}_{X_{\rho}} \lesssim A^kk!.
    \end{equation}
\end{proposition}
We recall the definitions of $Y$ and $X$ from \eqref{eq:def Y} and \eqref{eq:X simply denote}, respectively.
\begin{remark}
    The estimate \eqref{eq:psi N esti} is only a fixed-order statement. We do not track the dependence of the implicit constant on $N$ here, and $N$ is kept independent of $b$ throughout this subsection. The effective choice $N\sim |b|^{-1}$ and the corresponding exponentially small error will be discussed in the next subsection.
\end{remark}

\begin{remark}
    The role of the tail computation is different from that in Bourgain--Wang solutions \cite{KimKwonPark2025arXiv}. In the Bourgain--Wang setting, the parameter $b$ describes the Bourgain--Wang blow-up scaling, while the additional parameter $\eta$ is used to capture the unstable direction; accordingly, the relevant modulation law involves $b_s+\frac{b^2}{2}+\eta\approx 0$. In the present work, the target profile belongs to a different blow-up regime. We construct an almost self-similar profile, for which the scaling parameter $b$ should remain nearly frozen, and hence the desired modulation law is $b_s\approx 0$. Thus the tail computation is used to construct the negative-energy profile itself along this almost self-similar scaling direction.
\end{remark}

We emphasize that the mechanism behind the proof is not specific to the half-wave equation. 
In the present setting, for a complex-valued function $f$, set
\begin{align*}
    \Psi_b(f)\coloneqq |D|f+f-|f|^2f-ib\Lambda f.
\end{align*}
The crucial observation is the identity
\begin{align}\label{eq:Psi cancel}
    (\Psi_b(f),if)_r=0.
\end{align}
This identity and the evenness of $Q$ give the solvability condition \eqref{eq:solvable} at each step of the construction. Since it remains available after each update of the approximate profile, the induction can be continued to arbitrary order.

\begin{proof}[Proof of Proposition~\ref{prop: defining corrector profile infinitely}]
    \underline{\textbf{Proof of \eqref{eq:Qb_N_def} and \eqref{eq:psi N esti}.}}
    First, for a power series $F(b)=\sum_{n\ge 0}(ib)^n F_n$ with $F_n\in H^{1/2}$, we define
    \begin{equation}\label{eq:Coeff_def}
        \mathrm{Prof}_{n} F(b)\coloneqq F_n .
    \end{equation}
    Note that even though $F_n$ may be complex-valued, $\mathrm{Prof}_n$ is well-defined.
    We will use induction to construct $R_j$ satisfying \eqref{eq:psi N esti}.

    Assume that $R_1,\dots,R_{k-1}$ have already been chosen, and define the nonlinear term $\mathrm{NL}_k$ by
    \begin{equation*}
        \mathrm{NL}_k \coloneqq \mathrm{Prof}_{k}\bigl(|Q_b^{(k-1)}|^2Q_b^{(k-1)}\bigr).
    \end{equation*}
    We note that, for $0\leq j \leq 3k-3$,
    \begin{equation}\label{eq:profj relation}
        \mathrm{Prof}_{j}\bigl(|Q_b^{(k-1)}|^2Q_b^{(k-1)}\bigr)=
        \sum_{\substack{a+b+c=j\\ 0\le a,b,c\le k-1}}(-1)^c R_aR_bR_c 
    \end{equation}
    with a convention $R_0\coloneqq Q$. 
    The coefficient comparison at order $(ib)^k$ in the profile error gives the linear problem
    \begin{equation}\label{eq:b^k solver}
        L_Q[i^k R_k]= i^k\bigl(\Lambda R_{k-1}+\mathrm{NL}_k\bigr),
    \end{equation}
    with the convention $R_0\coloneqq Q$ and $\mathrm{NL}_1\coloneqq 0$.
    We solve \eqref{eq:b^k solver} inductively.

    We choose $R_1$ to be a real-valued solution of
    \begin{equation*}
        L_Q[iR_1]= i\Lambda Q.
    \end{equation*}
    Assume inductively that $R_1,\dots,R_{k-1}$ are real-valued and even.
    One can check that $\mathrm{NL}_k$ is real-valued and even. By the definition of $L_Q$, if a solution $R_k$ to \eqref{eq:b^k solver} exists, then $R_k$ is also real-valued and even. Moreover, by \eqref{eq:LQ invert}, we have $R_k \in X_{\rho^*}$.
    Therefore, the solvability condition \eqref{eq:solvable} reduces to
    \begin{align} \label{eq:solvable Rj}
        (i^k(\Lambda R_{k-1}+\mathrm{NL}_k), iQ)_r=0.
    \end{align}
    In fact, this is a direct consequence of \eqref{eq:Psi cancel}. Taking $f=Q_b^{(k-1)}$ in \eqref{eq:Psi cancel}, we get
    \begin{align*}
        0=(\Psi_b^{(k-1)},iQ_b^{(k-1)})_r=&(-(ib)^k(\Lambda R_{k-1}+\mathrm{NL}_k)+O(b^{k+1}),iQ+O(b))_r
        \\
        =&-b^{k}(i^k(\Lambda R_{k-1}+\mathrm{NL}_k), iQ)_r+O(b^{k+1}).
    \end{align*}
    Thus, dividing by $b^{k}$ and letting $b\to 0$, we conclude \eqref{eq:solvable Rj}.

    Therefore, we conclude the existence and uniqueness of $R_j$. The choice of $R_j$ removes the $b^j$ contribution in the profile error, which implies \eqref{eq:psi N esti}. More precisely, we have
    \begin{equation}\label{eq:Psi_tail_decomp}
        \Psi_b^{(N)} =-(ib)^{N+1}\Lambda R_N-\sum_{k=N+1}^{3N}(ib)^k\mathrm{Prof}_{k}\bigl(|Q_b^{(N)}|^2Q_b^{(N)}\bigr),
    \end{equation}
    which yields \eqref{eq:psi N esti}.
    
    \underline{\textbf{Proof of \eqref{eq:Rk_gevrey_Xrho0}.}}
    Since $R_1 \in X_{\rho^*}$ by \eqref{eq:Q in X} and \eqref{eq:LQ invert}, it suffices to prove \eqref{eq:Rk_gevrey_Xrho0} under the inductive assumption that $R_1,\cdots, R_{k-1} \in X_{\rho_0}$ for any $0<\rho_0\leq \frac{3\rho^*}{4}$.

    Let $\rho^*>0$ be as in Lemma~\ref{lem:Property of X_rho}.
    For $\rho\in(0,\rho^*]$, we define $a_k(\rho)\coloneqq \norm{R_k}_{X_\rho}$.
    From \eqref{eq:b^k solver} and \eqref{eq:LQ invert}, we obtain, for $k\ge2$,
    \begin{equation}\label{eq:ak_pre}
        a_k(\rho) \lesssim \norm{\Lambda R_{k-1}}_{X_\rho}+\norm{\NL_k}_{X_\rho}.
    \end{equation}
    Using the $\Lambda$-shift estimate \eqref{eq:shift_Cauchy_delta3}, for any $\delta\in(0,1]$ with $\rho+\delta\le\rho^*$,
    \begin{equation}\label{eq:ak_shift}
        \norm{\Lambda R_{k-1}}_{X_\rho}\leq \delta^{-1}\norm{R_{k-1}}_{X_{\rho+\delta}}
        =\delta^{-1}a_{k-1}(\rho+\delta).
    \end{equation}
    Moreover, by \eqref{eq: X_rho is Banach algebra}, \eqref{eq:inclusion for rho}, and \eqref{eq:Q in X}, there is a constant $C_*>0$ depending only on $\rho^*$ such that, for all $\rho\in(0,\rho^*]$,
    \begin{equation}\label{eq:NLk_bound}
        \norm{\NL_k}_{X_\rho} \lesssim \sum_{\substack{p+q=k\\ p,q\leq k-1}}        a_p(\rho)a_q(\rho) + \sum_{\substack{p+q+r=k\\ p,q,r\leq k-1}} a_p(\rho)a_q(\rho)a_r(\rho).
    \end{equation}
    Thus, combining \eqref{eq:ak_pre}--\eqref{eq:NLk_bound}, we arrive at
    \begin{equation}\label{eq:ak_master}
        a_k(\rho) \lesssim \delta^{-1}a_{k-1}(\rho+\delta) + \sum_{\substack{p+q=k\\ p,q\leq k-1}} a_p(\rho)a_q(\rho) + \sum_{\substack{p+q+r=k\\ a,b,c\leq k-1}}   a_p(\rho)a_q(\rho)a_r(\rho).
    \end{equation}
    Define
    \[
        A_k=A_k(\rho)\coloneqq \sup_{0<\rho<\rho^*} (\rho^*-\rho)^ka_k(\rho).
    \]
    For $\rho\in(0,\rho^*)$, fix $\delta=\frac{1}{k}(\rho^*-\rho)$.
    Then, one can check that $\rho+\delta<\rho^*$ for $k\geq 2$.
    We have
    \[
        (\rho^*-\rho)^{k-1}a_{k-1}(\rho+\delta) = (1-k^{-1})^{-k+1}(\rho^*-(\rho+\delta))^{k-1}a_{k-1}(\rho')
        \lesssim A_{k-1},
    \]
    which implies
    \begin{equation}\label{eq:linear_grand}
        (\rho^*-\rho)^k\delta^{-1}a_{k-1}(\rho+\delta)\ \lesssim\ kA_{k-1}.
    \end{equation}
    For the quadratic and cubic terms in \eqref{eq:ak_master}, by definition of $A_k$, we have
    \[
        (\rho^*-\rho)^{p+q} a_{p}(\rho)a_{q}(\rho)\le A_{p}A_{q}, \quad (\rho^*-\rho)^{p+q+r} a_p(\rho)a_q(\rho)a_r(\rho)\leq A_pA_qA_r.
    \]
    Multiplying \eqref{eq:ak_master} by $(\rho^*-\rho)^k$, taking the supremum on $\rho\in(0,\rho^*)$, and applying \eqref{eq:linear_grand}, we obtain
    \begin{equation}\label{eq:Ak_recurrence}
        A_k \lesssim\ kA_{k-1} + \sum_{\substack{p+q=k\\ p,q\leq k-1}} A_pA_q +
        \sum_{\substack{p+q+r=k\\ p,q,r\leq k-1}} A_pA_qA_r.
    \end{equation}
    Now, assume $A_j \leq M C^j j!$ for some $C>0$ for $j\in \{1,2,\cdots,k\}$. We have
    \[
        A_{k+1} \leq M C^\prime (k+1)! C^{k} + C^\prime\left(M^2 T(k+1) + M^3 S(k+1)\right) (k+1)!C^{k+1},
    \]
    where $T(k)$ and $S(k)$ are given by
    \[
        T(k) \coloneqq \sum_{\substack{p+q=k\\ p,q\leq k-1}}\frac{p!q!}{k!},\quad 
        S(k) \coloneqq \sum_{\substack{p+q+r=k\\ p,q,r\leq k-1}} \frac{p!q!r!}{k!}.
    \]
    Since $\sup_kT(k)+\sup_kS(k)<\infty$, choosing $M>0$ small so that the quadratic and cubic constants are absorbed, and then choosing $C$ large enough to handle the linear term and the initial value, \eqref{eq:Ak_recurrence} gives $A_{k+1}\le MC^{k+1}(k+1)!$.
    Therefore, for any $0<\rho_0 \le \tfrac{3\rho^*}{4}$, we obtain
    \[
        (\tfrac{\rho^*}{4})^ka_k(\rho_0)\le (\rho^*-\rho_0)^ka_k(\rho_0)\leq A_k\leq MC^kk!,
    \]
    which implies \eqref{eq:Rk_gevrey_Xrho0}.
\end{proof}

\subsection{Construction of the almost self-similar profile}\label{sec:almost self similar}
We now pass from the formal modified profile $Q_b^{(N)}$ to an actual almost self-similar profile $Q_b$. Here, the time-dependent choice $N\sim |b|^{-1}$ cannot be used directly since it would not give a $C^1$ family in $b$. We therefore define the profile by a truncated version of \textit{Borel's integral summation} method in the $\Lambda$-analytic space.

We set
\begin{equation}\label{eq: choice of parameter}
    \theta=\frac13,\qquad 
    \kappa = \frac{1}{2}\theta=\frac16,
    \qquad L(b)\coloneqq \frac{\theta}{A|b|},\qquad
    N(b)\coloneqq \Bigl\lfloor \frac{\kappa}{A|b|}\Bigr\rfloor,
\end{equation}
where $A>0$ is given in Proposition~\ref{prop: defining corrector profile infinitely}.
Define the Borel transform
\begin{equation}\label{eq:Qhat_def}
    \widetilde R(\xi,x)\coloneqq \sum_{j\ge1}\frac{(i\xi)^j}{j!}R_j(x),
    \qquad |\xi|<\frac1A,
\end{equation}
which converges absolutely in $X$ by \eqref{eq:Rk_gevrey_Xrho0}. Here $X$ is the $\Lambda$-analytic space in \eqref{eq:X simply denote}. 
We define the profile by the truncated Laplace transform
\begin{equation}\label{eq:Qb_tr_def}
    Q_b(x)  \coloneqq Q(x) + \int_0^{L(b)} e^{-\zeta}\widetilde R(b\zeta,x)\,d\zeta.
\end{equation}
The length $L(b)$ of the Laplace integral determines the effective truncation order of the formal expansion. Therefore, to reproduce the optimal truncation at order $N(b)\sim |b|^{-1}$ while keeping a smooth dependence on $b$, we choose $L(b)\sim |b|^{-1}$. Thus the truncation by the Laplace integral is a smoothed version of the sharp order cutoff $N(b)$.

Moreover, for $\zeta\in[0,L(b)]$, we have $|b\zeta|\le |b|L(b)<\frac{1}{A}$, and hence \eqref{eq:Qb_tr_def} is well-defined in $X$ by \eqref{eq:Qhat_def}. We also define
\[
    \Gamma(j+1,L)\coloneqq \int_0^{L} e^{-\zeta}\zeta^j\,d\zeta.
\]
By Fubini, inserting \eqref{eq:Qhat_def} into \eqref{eq:Qb_tr_def} gives
\begin{equation}\label{eq:Qb_tr_gamma_series}
    Q_b = Q+\sum_{j\ge1}(ib)^j\frac{\Gamma(j+1,L(b))}{j!}R_j .
\end{equation}

We now compare the almost self-similar profile $Q_b$ with the truncated profile $Q_b^{(N)}$ introduced in \eqref{eq:Qb_N_def}. We also denote 
\begin{equation}\label{eq:partial_b QBN}
    \partial_b Q_b^{(N)}\coloneqq \sum_{j=1}^{N}\partial_b[(ib)^jR_j]=i\sum_{j=1}^{N}j (ib)^{j-1}R_j.
\end{equation}
That is, $\partial_b Q_b^{(N)}$ denotes the partial derivative with respect to $b$ while treating $N$ as a fixed parameter. The dependence $N=N(b)$ is imposed only after this differentiation.

The next lemma quantifies this approximation in $X$. 

\begin{lemma}\label{lem:Qb property}
    Let $0<\rho<\frac{3\rho^*}{4}$. There exist $b^*>0$ and $c>0$ such that for all $0<|b|<b^*$, with $N=N(b)$ given by \eqref{eq: choice of parameter}, the following hold:
    \begin{enumerate}
        \item (Well-definedness) We have
        \begin{equation} \label{eq:Qb X bound}
            \norm{Q_b}_{X_\rho}+\norm{Q_b^{(N)}}_{X_\rho}  \lesssim_{b^*} 1.
        \end{equation}
    
        \item (Approximation between $Q_b$ and $Q_b^{(N)}$) We have
        \begin{equation} \label{eq:Qtr_minus_QN_final}
            \norm{Q_b-Q_b^{(N)}}_{X_\rho} \lesssim e^{-\frac{c}{|b|}}.
        \end{equation}
        Moreover, we have
        \begin{equation}\label{eq:dQtr_minus_dQN_final}
            \norm{\partial_b Q_b-\partial_b Q_b^{(N)}}_{X_\rho}
            \lesssim e^{-\frac{c}{|b|}}.
        \end{equation}

        \item (Approximation to $Q$ and control of $\partial_b Q_b$) We have
        \begin{equation}\label{eq:Qb_to_Q_bound}
            \norm{Q_b-Q}_{X_\rho}\lesssim |b|,
        \end{equation}
        Moreover, we have
        \begin{equation}\label{eq:dQb_to_iR1_bound}
            \norm{\jp y^2(\partial_b Q_b-iR_1)}_{L^\infty}\lesssim |b|.
        \end{equation}

        \item (Profile error for $Q_b^{(N)}$) We have
        \begin{equation} \label{eq:PsibN esti}
            \norm{\Psi_b^{(N)}}_{X_\rho} \lesssim e^{-\frac{c}{|b|}}.
        \end{equation}
    \end{enumerate}
\end{lemma}
We defer the proof of this lemma to the end of this section. As a consequence of the lemma, we obtain the desired exponential smallness of the profile error associated with $Q_b$.
\begin{proposition}\label{prop:selfsimilar error}
    Let $\Psi_b$ be a profile error for $Q_b$:
    \begin{equation}\label{eq:Psi b def}
        \Psi_b \coloneqq |D|Q_b+Q_b-|Q_b|^2Q_b - ib\Lambda Q_b.
    \end{equation}
    Then, for $b \in (0,b^*)$, we have
    \begin{equation}\label{eq:Psi_exp_small_H1_conclusion}
        \norm{\Psi_b}_{X^{(H^1)}} \lesssim e^{-\frac{c}{|b|}},
    \end{equation}
    where $c>0$ is given in Lemma~\ref{lem:Qb property}.
    Moreover, we have
    \begin{align}\label{eq:LQb Lambda Qb eq}
        L_{Q_b}\Lambda Q_{b}=(1+\Lambda )\Psi_b -Q_b+ ib\Lambda Q_b+ib\Lambda^2 Q_b.
    \end{align}
\end{proposition}
\begin{proof}[Proof of Proposition ~\ref{prop:selfsimilar error} assuming Lemma~\ref{lem:Qb property}]
    Let $N=N(b)$ and set $w=Q_b-Q_b^{(N)}$. Fix $\rho_1\in(\rho^*/2,3\rho^*/4)$. By \eqref{eq:Qtr_minus_QN_final}, applied with $Y=H^1,H^2$ at radius $\rho_1$, together with \eqref{eq:shift_Cauchy_delta3} and \eqref{eq:derivative loss of X_rho}, we have
    \[
        \norm{w}_{X^{(H^1)}}+\norm{\Lambda w}_{X^{(H^1)}}+\norm{|D|w}_{X^{(H^1)}}\lesssim e^{-\frac{c}{|b|}}.
    \]
    The algebra estimate in $X^{(H^1)}$ gives the local Lipschitz bound
    \[
        \norm{|Q_b|^2Q_b-|Q_b^{(N)}|^2Q_b^{(N)}}_{X^{(H^1)}}
        \lesssim (\norm{Q_b}_{X^{(H^1)}}^2+\norm{Q_b^{(N)}}_{X^{(H^1)}}^2)\norm{w}_{X^{(H^1)}}
        \lesssim e^{-\frac{c}{|b|}},
    \]
    where we used \eqref{eq:Qb X bound}. Hence, we obtain
    \[
        \norm{\Psi_b-\Psi_b^{(N)}}_{X^{(H^1)}}\lesssim e^{-\frac{c}{|b|}}.
    \]
    Combining this with \eqref{eq:PsibN esti} gives \eqref{eq:Psi_exp_small_H1_conclusion}.

    From \eqref{eq:Psi_exp_small_H1_conclusion}, we have
    \begin{align*}
        |D|[Q_b]_{\lambda^{-1}}+\lambda[Q_b]_{\lambda^{-1}}-|[Q_b]_{\lambda^{-1}}|^2[Q_b]_{\lambda^{-1}} 
        =\lambda [\Psi_b + ib\Lambda Q_b]_{\lambda^{-1}}.
    \end{align*}
    Thus, differentiating this equation with respect to $\lambda$ at $\lambda=1$, we get
    \begin{equation*}
        L_{Q_b}\Lambda Q_{b}=(1+\Lambda )\Psi_b -Q_b+ ib\Lambda Q_b+ib\Lambda^2 Q_b. \qedhere
    \end{equation*}
\end{proof}

We conclude this section by proving Lemma~\ref{lem:Qb property}.
\begin{proof}[Proof of Lemma~\ref{lem:Qb property}]
    Let $L=L(b)$ and $N=N(b)$ be given by \eqref{eq: choice of parameter}, and set
    \begin{equation*}
        a_j(b)\coloneqq \tfrac{1}{j!}\Gamma(j+1,L).
    \end{equation*}
    For $b^*>0$ small enough, we have $N+1\leq \frac34L$ and $A|b|N\leq\kappa$. We first record the coefficient estimates
    \begin{align}
        &\sum_{j=1}^{N}|b|^j|a_j(b)-1|A^j j! +\sum_{j\ge N+1}|b|^j a_j(b)A^j j!
        \lesssim e^{-\frac{c}{|b|}}, \label{eq:borel_coeff_0}
        \\
        &\sum_{j=1}^{N}j|b|^{j-1}|a_j(b)-1|A^j j! +\sum_{j\ge N+1}j|b|^{j-1}a_j(b)A^j j!
        +\sum_{j\ge1}|b|^j|a_j'(b)|A^j j!
        \lesssim e^{-\frac{c}{|b|}}.\label{eq:borel_coeff_1}
    \end{align}
    Indeed,
    \begin{equation*}
        0\le a_j(b)\le 1,\qquad
        1-a_j(b)=e^{-L}\sum_{m=0}^{j}\frac{L^m}{m!}.
    \end{equation*}
    Thus, for $j\le N\le \frac34L$, Lemma~\ref{lem:poisson_left_tail} gives $|a_j(b)-1|\lesssim e^{-c_1L}$. Also, by \eqref{eq:Stirling} and $A|b|j\leq\kappa$ for $j\leq N$,
    \begin{equation*}
        \sum_{j=1}^{N}(A|b|)^jj!\lesssim \sqrt{N}\sum_{j\ge1}\Bigl(\frac{\kappa}{e}\Bigr)^j\lesssim \sqrt{N},
        \qquad
        \sum_{j=1}^{N}j(A|b|)^jj!\lesssim N^{3/2}.
    \end{equation*}
    The corresponding terms in \eqref{eq:borel_coeff_0}--\eqref{eq:borel_coeff_1} are therefore exponentially small after reducing $c>0$; in the second estimate the extra factor $|b|^{-1}$ is only polynomial in $L$. For the tails, we use
    \begin{equation*}
        a_j(b)\leq \tfrac{1}{(j+1)!}L^{j+1},\qquad A|b|L=\theta.
    \end{equation*}
    This gives
    \begin{equation*}
        \sum_{j\ge N+1}|b|^j a_j(b)A^jj!
        \lesssim \frac{1}{A|b|}\sum_{j\ge N+1}\frac{\theta^{j+1}}{j+1}
        \lesssim e^{-\frac{c}{|b|}},
    \end{equation*}
    and
    \begin{equation*}
        \sum_{j\ge N+1}j|b|^{j-1}a_j(b)A^jj!
        \lesssim |b|^{-2}\sum_{j\ge N+1}j\theta^j
        \lesssim e^{-\frac{c}{|b|}}.
    \end{equation*}
    Finally, since $\partial_L\Gamma(j+1,L)=e^{-L}L^j$ and $|\partial_b L|\lesssim |b|^{-2}$,
    \begin{equation*}
        \sum_{j\ge1}|b|^j|a_j'(b)|A^jj!
        \lesssim e^{-L}|\partial_bL|\sum_{j\ge1}(A|b|L)^j
        \lesssim e^{-\frac{c}{|b|}}.
    \end{equation*}

    We now prove \eqref{eq:Qtr_minus_QN_final} and \eqref{eq:dQtr_minus_dQN_final}. From \eqref{eq:Qb_tr_gamma_series} and \eqref{eq:Qb_N_def},
    \begin{equation*}
        Q_b-Q_b^{(N)}
        =\sum_{j=1}^{N}(ib)^j(a_j(b)-1)R_j
        +\sum_{j\ge N+1}(ib)^ja_j(b)R_j.
    \end{equation*}
    Hence, \eqref{eq:borel_coeff_0} and \eqref{eq:Rk_gevrey_Xrho0} imply \eqref{eq:Qtr_minus_QN_final}. For \eqref{eq:dQtr_minus_dQN_final}, we have
    \begin{equation*}
        \begin{aligned}
            \partial_b Q_b-\partial_b Q_b^{(N)}
            &=\sum_{j=1}^{N}i(ib)^{j-1}j(a_j(b)-1)R_j
            +\sum_{j\ge N+1}i(ib)^{j-1}ja_j(b)R_j \\
            &\quad+\sum_{j\ge1}(ib)^ja_j'(b)R_j.
        \end{aligned}
    \end{equation*}
    Here, we recall the definition of $\partial_b Q_b^{(N)}$ in \eqref{eq:partial_b QBN}.
    Combining this identity with \eqref{eq:borel_coeff_1} and \eqref{eq:Rk_gevrey_Xrho0} gives \eqref{eq:dQtr_minus_dQN_final}.

    \underline{\textbf{Step 3.}} Now, we prove \eqref{eq:Qb_to_Q_bound} and \eqref{eq:dQb_to_iR1_bound}. By \eqref{eq:Qb_tr_def}, we have
    \begin{equation}\label{eq:Qb_minus_Q_int}
        Q_b-Q = {\textstyle\int_0^{L}} e^{-\zeta} \widetilde R(b\zeta)\,d\zeta.
    \end{equation}
    Since the Borel series \eqref{eq:Qhat_def} starts at $j=1$, we have 
    \begin{equation}\label{eq:Qhat_factor_tau}
        \widetilde R(\xi)=\xi \widetilde R_1(\xi),
        \qquad
        \widetilde R_1(\xi)= \widetilde R_1(\xi,x)\coloneqq \sum_{j\ge1}\tfrac{1}{j!} i^j \xi^{ j-1} R_j(x).
    \end{equation}
    We claim that $\widetilde R_1$ is uniformly bounded on $|\xi|\le \theta/A$. Indeed, for $|\xi|\le \theta/A$, by \eqref{eq:Rk_gevrey_Xrho0} and \eqref{eq:Qhat_factor_tau}, we have
    \begin{equation}\label{eq:Qhat1_uniform_bound}
        \norm{\widetilde R_1(\xi,\cdot)}_{X_\rho} \le \sum_{j\ge1}\tfrac{1}{j!}|\xi|^{j-1} \norm{R_j}_{X_\rho}
        \lesssim \sum_{j\ge1}A^j|\xi|^{j-1} \lesssim 1 .
    \end{equation}
    Now, for $\zeta\in[0,L]$, we have $|b\zeta|\le \theta/A$. Thus, for $\xi=b\zeta$, combining
    \eqref{eq:Qhat_factor_tau} and \eqref{eq:Qhat1_uniform_bound}, we have
    \begin{equation}\label{eq:Qhat_bt_bound}
        \norm{\widetilde R(b\zeta,\cdot)}_{X_\rho} = |b\zeta| \norm{\widetilde R_1(b\zeta,\cdot)}_{X_\rho} \lesssim |b| \zeta.
    \end{equation}
    Inserting \eqref{eq:Qhat_bt_bound} into \eqref{eq:Qb_minus_Q_int}, we obtain
    \begin{equation*}
        \norm{Q_b-Q}_{X_\rho} \lesssim |b|{\textstyle\int_0^{L}} \zeta e^{-\zeta}\,d\zeta
        \le  |b|,
    \end{equation*}
    which proves \eqref{eq:Qb_to_Q_bound}.

    We next prove \eqref{eq:dQb_to_iR1_bound}. By taking $Y=\langle y\rangle^{-2}L^\infty$ in \eqref{eq:dQtr_minus_dQN_final}, we have
    \begin{equation}\label{eq:partial_bQ_b - partial_bQ_b^{(N)} <b}
        \norm{\jp y^2(\partial_b Q_b-\partial_b Q_b^{(N)})}_{L^\infty}
        \leq
        \norm{\partial_b Q_b-\partial_b Q_b^{(N)}}_{X}
        \lesssim e^{-\frac{c}{|b|}}\lesssim |b|.
    \end{equation}
    On the other hand, from \eqref{eq:partial_b QBN} and \eqref{eq:Rk_gevrey_Xrho0} with \eqref{eq:Stirling} we have
    \begin{equation}\label{eq:partial_bQ_b^{(N)} - iR_1 <b}
        \norm{\jp y^{2}(\partial_b Q_b^{(N)}-iR_1)}_{L^\infty}
        \lesssim
        \sum_{j=2}^{N}j|b|^{j-1}A^j j! \lesssim |b| \sum_{j=2}^{N} \left(\frac{A|b|j}{e}\right)^{j-2} \lesssim |b|.
    \end{equation}
    In the last inequality, we used $A|b|j\leq A|b|N\leq \kappa$ for $j\leq N$, and $\frac{\kappa}{e}<1$. Combining \eqref{eq:partial_bQ_b - partial_bQ_b^{(N)} <b} and \eqref{eq:partial_bQ_b^{(N)} - iR_1 <b} we conclude \eqref{eq:dQb_to_iR1_bound}.

    \underline{\textbf{Step 4.}} Finally, we show \eqref{eq:PsibN esti}.
    We recall $\Psi_b^{(N)}$ in \eqref{eq:Psi_tail_decomp}.
    Using \eqref{eq: X_rho is Banach algebra} and \eqref{eq:profj relation} with \eqref{eq:Rk_gevrey_Xrho0}, for $k\geq N+1$, we obtain
    \begin{equation}\label{eq:NLk_gevrey_bound}
        \norm{\mathrm{Prof}_k}_{X_{\rho}} \lesssim
        \sum_{\substack{a+b+c=k}} \norm{R_a}_{X_{\rho}}\norm{R_b}_{X_{\rho}}
        \norm{R_c}_{X_{\rho}} \lesssim A^k k!.
    \end{equation}
    Moreover, by \eqref{eq:shift_Cauchy_delta3} and \eqref{eq:inclusion for rho}, fixing $\delta>0$ so that $\rho+\delta=\frac{3\rho^*}{4}$, we have
    \begin{equation}\label{eq:LambdaRN_bound}
        \norm{\Lambda R_N}_{X_{\rho}} \leq \delta^{-1}\norm{R_N}_{X_{\rho+\delta}}
        \leq \delta^{-1}\norm{R_N}_{X_{3\rho^*/4}} \lesssim_{\rho} A^N N!.
    \end{equation}
    Combining \eqref{eq:Psi_tail_decomp}, \eqref{eq:NLk_gevrey_bound}, and \eqref{eq:LambdaRN_bound}, we get
    \begin{equation*}
         \norm{\Psi_b^{(N)}}_{X_{\rho}} \lesssim |b|^{N+1}\norm{\Lambda R_N}_{X_{\rho}} +\sum_{k=N+1}^{3N}|b|^k\norm{\mathrm{Prof}_k}_{X_{\rho}}
        \lesssim_{\rho} \sum_{k=N+1}^{3N}(A|b|)^k k!.
    \end{equation*}
    By \eqref{eq: choice of parameter}, we have $3A|b|N \leq \frac{3}{2}\theta<1$, and
    \begin{equation*}
        (A|b|)^{k+1} (k+1)!< (\tfrac{3}{2}\theta) (A|b|)^k k!,\quad \text{for}\quad N\leq k\leq 3N.
    \end{equation*}
    This means that
    \begin{equation*}
        \sum_{k=N}^{3N}(A|b|)^k k! \lesssim (A|b|)^N N! 
        \le (A|b|N)^N\le 2^{-N} \lesssim e^{-\frac{c}{|b|}},
    \end{equation*}
    which concludes \eqref{eq:PsibN esti}.
\end{proof}

\section{Modulation analysis}
In this section, we prove Theorem~\ref{thm:main thm}. The argument follows the modulation--virial scheme of Merle--Rapha\"el \cite{MerleRaphael2003GAFA}. We first obtain a sharp decomposition of the solution near the almost self-similar profile $Q_b$. This fixes the parameters $\lambda$, $b$, and $\gamma$ through the orthogonality conditions. We then pass to renormalized variables, derive the equation for the radiation, and obtain the modulation estimates. The localized virial argument is carried out in the same spirit as in the NLS case; the spectral estimate \eqref{eq:positivity of H}, proved in \cite{Park2026arXiv}, is used there to estimate the quadratic form associated with the half-wave local virial functional.

Let $u(t)$ be a solution to \eqref{eq:half-wave}, with 
\begin{equation*}
    \alpha_0 \coloneqq \norm{u_0}_{L^2}^2 - \norm{Q}_{L^2}^2,
    \qquad E_0 \coloneqq E(u_0)<0.
\end{equation*}
We recall that $u(t)$ blows up in finite time \cite{Park2026arXiv} when the initial data satisfies $0<\alpha_0<\alpha^\prime$ and $E(u_0)<0$ for some small $\alpha^\prime>0$.
Let $T>0$ be a forward blow-up time of $u(t)$. We start with the sharp decomposition.

\begin{lemma}[Decomposition]\label{lem:decomposition}
    There exists $\alpha^*>0$ such that, for all $0<\alpha_0 <\alpha^*$, there are continuous parameters $(\lambda,\gamma,b) : [0,T)\to (0,\infty)\times \bbR^2$ such that
    \[
        \epsilon(t,y) \coloneqq e^{-i\gamma(t)}\lambda^{1/2}(t)u(t,\lambda(t)y) - Q_{b(t)}(y)
    \]
    satisfies the orthogonality conditions
    \begin{equation}\label{eq: orthogonality conditions for epsilon}
        (\epsilon,\Lambda Q_b)_r = (\epsilon, i\Lambda Q_b)_r = (\epsilon, i\Lambda^2 Q_b)_r = 0,
    \end{equation}
    and the estimates
    \begin{equation}\label{eq:mod para smallness}
        \left|1-\lambda(t)\frac{\norm{u}_{\dot H^{1/2}}^2}{\norm{Q}_{\dot H^{1/2}}^2}\right| + \norm{\epsilon}_{H^{1/2}} + |b(t)| \leq \delta(\alpha_0) \to 0,\text{ as} \quad  \alpha_0 \to 0.
    \end{equation}
\end{lemma}
The proof of Lemma~\ref{lem:decomposition} can be found in Section~\ref{sec:decom append}.

Based on this decomposition, we introduce the renormalized variables by writing
\begin{equation}\label{eq:decompose}
    u(t,x)=\lambda(t)^{-\frac12}v(s,y)e^{i\gamma(t)},
    \qquad v(s,y)=Q_{b(s)}(y)+\epsilon(s,y),
\end{equation}
where
\begin{equation*}
    y=\frac{x}{\lambda(t)},\qquad \frac{ds}{dt}=\frac1{\lambda(t)},
    \qquad \Lambda=\frac12+y\partial_y.
\end{equation*}
Then $v$ satisfies
\begin{equation}\label{eq:v_equation}
    i v_s - |D| v + v|v|^2 - \gamma_s v - i\frac{\lambda_s}{\lambda}\Lambda v=0.
\end{equation}
We set $\td\gamma_s \coloneqq \gamma_s-1$.
Substituting $v=Q_b+\epsilon$ into \eqref{eq:v_equation} gives
\begin{equation}\label{eq:epsilon flow}
    \begin{aligned}
        &\partial_s\epsilon + iL_{Q_b}[\epsilon] + b\Lambda \epsilon 
        \\
        &=-b_s\partial_bQ_b+\left(\frac{\lambda_s}{\lambda}+b\right)\Lambda(Q_b+\epsilon) - i\td\gamma_s(Q_b+\epsilon) + i \mathcal{N}(\epsilon) -i\Psi_b,
    \end{aligned}
\end{equation}
where
\begin{equation*}
     \mathcal{N}(\epsilon) \coloneqq
     2\Re\{\ol{Q_b} \epsilon\}\,\epsilon + |\epsilon|^2 Q_b + |\epsilon|^2 \epsilon.
\end{equation*}
Then, we have the first modulation estimates.

\begin{lemma}
    For given $c>0$ in Lemma~\ref{lem:Qb property}, there exists $\alpha^*>0$ so that for $\alpha_0 \in (0,\alpha^*)$ and $(\lambda(s),\gamma(s),b(s))$ are $C^1$-function of $s$ in $\bbR$, and we have the following properties:
    \begin{itemize}
        \item (Estimate induced by conservation laws) We have
        \begin{equation}\label{eq:estimate from conserv}
            \left|\lambda |E_0| - (Q_b,\epsilon)_r\right| \lesssim \norm{\epsilon}_{\mathcal{H}^{1/2}}^2 + e^{-\frac{c}{|b|}}.
        \end{equation}
    
        \item (Estimate on the modulation parameters) We have
        \begin{equation}\label{eq:first law}
            \left|\frac{\lambda_s}{\lambda}+b\right| + |\td \gamma_s| + |b_s| \lesssim \norm{\epsilon}_{\mathcal{H}^{1/2}} + e^{-\frac{c}{|b|}}.
        \end{equation}
    \end{itemize}
\end{lemma}

\begin{proof}
    Combining the energy conservation law with $E_0<0$, \eqref{eq:Psi b def}, Pohozaev identity $E(Q_b)=(\Psi_b,\Lambda Q_b)_r$, \eqref{eq:Hardy inequ}, and Gagliardo--Nirenberg, we have
    \begin{equation}\label{eq:energy conserv identity}
        \begin{aligned}
            -&\lambda|E_0| = E(Q_b+\epsilon) 
            \\
            =& E(Q_b) + (|D|Q_b-|Q_b|^2Q_b,\epsilon)_r +N_{2,E}- \tfrac{1}{4} {\textstyle\int} 4\Re(Q_b\overline{\epsilon})|\epsilon|^2 + |\epsilon|^4 dy
            \\
            =& (\Psi_b,\epsilon+\Lambda Q_b)_r - (Q_b - ib\Lambda Q_b,\epsilon)_r+N_{2,E} +\calO(\sqrt{\alpha_0}\|\epsilon\|_{\calH^{1/2}}),
        \end{aligned}
    \end{equation}
    where
    \begin{equation}\label{eq:eps quad ener def}
        \begin{aligned}
            N_{2,E}\coloneqq &
            \tfrac{1}{2}\norm{\epsilon}_{\dot H^{1/2}}^2 - \tfrac{1}{4}{\textstyle\int} 4\Re\{Q_b\overline{\epsilon}\}^2 + 2|Q_b|^2|\epsilon|^2  dy
            =\mathcal{O}(\norm{\epsilon}_{\mathcal{H}^{1/2}}^2).
        \end{aligned}
    \end{equation}
    Since $\norm{\Psi_b}_{L^2}\lesssim e^{-\frac{c}{|b|}}$ and $(i\Lambda Q_b,\epsilon)_r=0$, we obtain \eqref{eq:estimate from conserv}.

    Now, we prove \eqref{eq:first law}.
    Note that from \eqref{eq:Qb_to_Q_bound} and \eqref{eq:dQb_to_iR1_bound}, we have
    \[
        \Lambda^j Q_b = \Lambda^j Q + \mathcal{O}(b\langle y \rangle^{-2}),\quad j=0,1,2,
        \qquad \partial_b Q_b = iR_1 + \mathcal{O}(b\langle y \rangle^{-2}).
    \]
    Differentiating $(\epsilon,\Lambda Q_b)_r$ with respect to $s$, we have
    \begin{align*}
        0&=\partial_s(\epsilon,\Lambda Q_b)_r = (\partial_s \epsilon, \Lambda Q_b)_r + b_s(\epsilon,\Lambda \partial_bQ_b)_r.
    \end{align*}
    Using \eqref{eq:epsilon flow}, \eqref{eq:Hardy inequ}, \eqref{eq:mod para smallness}, and $Q_b \in  X^{(H^1)} \cap X^{(\langle y\rangle^{-2}L^\infty)}$, we have
    \begin{equation*}
        \begin{aligned}
            &(\partial_s \epsilon, \Lambda Q_b)_r 
            \\
            &= -b_s(\partial_b Q_b,\Lambda Q_b)_r + \left(\tfrac{\lambda_s}{\lambda} + b\right)(\Lambda Q_b + \Lambda \epsilon,\Lambda Q_b) - \td \gamma_s(i(Q_b+\epsilon),\Lambda Q_b)_r 
            \\
            & \quad - (iL_{Q_b}[\epsilon],\Lambda Q_b)_r - b(\Lambda \epsilon, \Lambda Q_b)_r
            + \mathcal{O}(\norm{\epsilon}_{\mathcal{H}^{1/2}}+e^{-\frac{c}{|b|}}) \\
            &= \left(\tfrac{\lambda_s}{\lambda}+b\right)\norm{\Lambda Q}_{L^2}^2  + \mathcal{O}(\delta(\alpha_0)(|b_s|+ \left|\tfrac{\lambda_s}{\lambda}+b\right|+|\td\gamma_s|)+\norm{\epsilon}_{\mathcal{H}^{1/2}} + e^{-\frac{c}{|b|}}),
        \end{aligned}
    \end{equation*}
    and we have
    \[
        b_s(\epsilon,\Lambda \partial_b Q_b)_r =\mathcal{O}(|b_s|\sqrt{\alpha_0}).
    \]
    Note that $(iL_{Q_b}[\epsilon],\Lambda Q_b)_r - b(\Lambda \epsilon, \Lambda Q_b)_r=\calO (\norm{\epsilon}_{\mathcal{H}^{1/2}})$ by \eqref{eq:Hardy inequ}.
    Therefore, after reducing $\alpha^*>0$ if necessary, we have
    \begin{equation}\label{eq:lambda_s/lambda+b estimate}
        \left|\frac{\lambda_s}{\lambda}+b\right| \lesssim \delta(\alpha_0)(|b_s|+|\td \gamma_s|) + \norm{\epsilon}_{\mathcal{H}^{1/2}}+ e^{-\frac{c}{|b|}} .
    \end{equation}
    For $(\epsilon,i\Lambda Q_b)_r$, by a similar argument with \eqref{eq:lambda_s/lambda+b estimate}, we get
    \begin{equation*}
        0=\partial_s(\epsilon,i\Lambda Q_b)_r = (\partial_s\epsilon, i\Lambda Q_b)_r + b_s(\epsilon,i \Lambda \partial_b Q_b)_r
        =(\partial_s\epsilon, i\Lambda Q_b)_r+\calO(|b_s|\sqrt{\alpha_0}),
    \end{equation*}
    and
    \begin{equation}\label{eq:b_s identity}
        \begin{aligned}
            (\partial_s&\epsilon, i\Lambda Q_b)_r 
            \\
            =& -b_s(\partial_b Q_b, i\Lambda Q_b)_r + \left(\tfrac{\lambda_s}{\lambda}+b\right)(\Lambda Q_b+\Lambda \epsilon,i\Lambda Q_b)_r - \td \gamma_s(Q_b+\epsilon,\Lambda Q_b)_r \\
            & -(L_{Q_b}\epsilon, \Lambda Q_b)_r -b(\Lambda\epsilon, i\Lambda Q_b)_r 
            +(\mathcal{N}, \Lambda Q_b)_r-(\Psi_b,\Lambda Q_b)_r
            \\
            =& -b_s e_1  + \mathcal{O}(\delta(\alpha_0)\left(|b_s|+\left|\tfrac{\lambda_s}{\lambda}+b\right|+|\td \gamma_s|\right) + \norm{\epsilon}_{\mathcal{H}^{1/2}}+ e^{-\frac{c}{|b|}}),
        \end{aligned}
    \end{equation}
    where $e_1= (R_1,\Lambda Q)_r > 0$. Therefore, we obtain
    \begin{equation}\label{eq:b_s estimate}
        |b_s|\lesssim  \delta(\alpha_0)\left(\left|\tfrac{\lambda_s}{\lambda}+b\right| + |\td \gamma_s|\right) + \norm{\epsilon}_{\mathcal{H}^{1/2}} + e^{-\frac{c}{|b|}}.
    \end{equation}
    Finally, differentiating $(\epsilon,i\Lambda^2 Q_b)_r=0$ and using the same estimates gives
    \begin{align*}
        0
        =& \td \gamma_s\|\Lambda Q\|_{L^2}^2
        -b_s (R_1,\Lambda^2 Q)_r \\
        &+\mathcal{O}(\delta(\alpha_0)\left(|b_s|+\left|\tfrac{\lambda_s}{\lambda}+b\right|+|\td \gamma_s|\right)
        + \norm{\epsilon}_{\mathcal{H}^{1/2}}+ e^{-\frac{c}{|b|}}),
    \end{align*}
    whence
    \begin{equation}\label{eq:gamma_s estimate}
        |\td \gamma_s|
        \lesssim |b_s|+\delta(\alpha_0)\left|\tfrac{\lambda_s}{\lambda}+b\right|
        +\norm{\epsilon}_{\mathcal{H}^{1/2}}+e^{-\frac{c}{|b|}}.
    \end{equation}
    Therefore, combining \eqref{eq:lambda_s/lambda+b estimate}, \eqref{eq:b_s estimate}, and \eqref{eq:gamma_s estimate}, we conclude \eqref{eq:first law}.
\end{proof}

In view of the first modulation law \eqref{eq:first law}, it remains to control the evolution of $b$. To this end, we derive a second modulation law from the virial identity
\begin{equation}\label{eq:virial identity}
    \partial_t\Phi(u)=2E(u_0),
\end{equation}
where $\Phi(u)=(iu,\Lambda u)_r$.
Using the decomposition \eqref{eq:decompose}, we expand \eqref{eq:virial identity} in the renormalized variables. The orthogonality conditions \eqref{eq: orthogonality conditions for epsilon} eliminate the leading cross terms between the profile and the radiation, yielding
\begin{align*}
    \partial_s\Phi(\epsilon)=b_s(2e_1+O(b))-2\lambda(s)|E_0|, \qquad e_1=(R_1,\Lambda Q)_r>0.
\end{align*}
Thus, the control of $b_s$ is reduced to a lower bound for the radiation term $\partial_s\Phi(\epsilon)$. At the linearized level around $Q_b$, the radiation $\epsilon$ satisfies the linearized flow $i\partial_s\epsilon=L_{Q_b}\epsilon$, and hence
\begin{align*}
    \tfrac12\partial_s\Phi(\epsilon)=\tfrac12([L_{Q_b},\Lambda]\epsilon,\epsilon)_r
    \approx \tfrac12([L_{Q},\Lambda]\epsilon,\epsilon)_r=\mathbf H(\epsilon).
\end{align*}
The spectral property of $\mathbf H$, \eqref{eq:positivity of H} then yields the following lower bound for $b_s$.

\begin{proposition} We have
    \begin{equation}\label{eq:b_s coercivity}
        b_s\gtrsim \lambda|E_0|+ \|\epsilon\|_{\calH^{1/2}}^2-e^{-\frac{c}{|b|}}.
    \end{equation}
\end{proposition}
\begin{proof}
Using the first identity of \eqref{eq:b_s identity}, the orthogonality $0=\partial_s(\epsilon,i\Lambda Q_b)_r$, and \eqref{eq: orthogonality conditions for epsilon}, we have
    \begin{align*}
        b_s (e_1+\calO(\delta(\alpha_0)))=&-(\epsilon, L_{Q_b}\Lambda Q_b- ib\Lambda^2 Q_b)_r 
        \\
        &+(\mathcal{N}_2, \Lambda Q_b)_r+(|\epsilon|^2\epsilon, \Lambda Q_b)_r-(\Psi_b,\Lambda Q_b)_r,
    \end{align*}
    where
    \begin{align*}
        \mathcal{N}_2\coloneqq 2\Re\{\ol{Q_b} \epsilon\}\,\epsilon + |\epsilon|^2 Q_b.
    \end{align*}
    Here, we used $(\Lambda Q_b,i\Lambda Q_b)_r=(Q_b,\Lambda Q_b)_r=0$. Thanks to \eqref{eq:LQb Lambda Qb eq} and \eqref{eq:energy conserv identity}, we get
    \begin{align*}
        (\epsilon, L_{Q_b}\Lambda Q_b- ib\Lambda^2 Q_b)_r
        =-(\epsilon,Q_b- ib\Lambda Q_b)_r
        =-\lambda|E_0|-(\Psi_b,\epsilon+\Lambda Q_b)_r -N_{2,E},
    \end{align*}
    where $N_{2,E}$ is given in \eqref{eq:eps quad ener def}. Thus, using \eqref{eq:Psi_exp_small_H1_conclusion} and Gagliardo--Nirenberg with \eqref{eq:Hardy inequ}, we obtain
    \begin{equation}\label{eq:b_s control pf}
        b_s (e_1+\calO(|b|+\sqrt{\alpha_0}))=\lambda|E_0|+(\mathcal{N}_2, \Lambda Q_b)_r+N_{2,E}+\calO(\sqrt{\alpha_0}\|\epsilon\|_{\calH^{1/2}}^2 +e^{-\frac{c}{|b|}}).
    \end{equation}
    Moreover, we have
    \begin{align*}
        (\mathcal{N}_2, \Lambda Q_b)_r+N_{2,E}=\tfrac12 ([L_{Q_b},\Lambda]\epsilon,\epsilon)_r=\mathbf{H}+\td{\mathbf{H}}_b,
    \end{align*}
    where
    \begin{align*}
        \mathbf{H}\coloneqq \tfrac12 ([L_Q,\Lambda]\epsilon,\epsilon)_r, \quad
        \td{\mathbf{H}}_b\coloneqq \tfrac12 ([L_{Q_b}-L_Q,\Lambda]\epsilon,\epsilon)_r.
    \end{align*}
    Note that
    \begin{align*}
        ([L_v,\Lambda]\epsilon,\epsilon)_r =2\Re {\int} |\epsilon|^2 \ol{v}\cdot y\partial_y v dy+4{\int} \Re(\ol v \epsilon)\Re(y \partial_y\ol{v} \cdot \epsilon) dy.
    \end{align*}
    Since $Q_b-Q=\calO(b \langle y \rangle^{-2})$ and $\Lambda(Q_b-Q)=\calO(b \langle y \rangle^{-2})$, we have
    \begin{align}\label{eq:bfH_b error}
        |\td{\mathbf{H}}_b|\lesssim |b|\|\epsilon\|_{\calH^{1/2}}^2. 
    \end{align}
    Thanks to \eqref{eq:positivity of H} with \eqref{eq: orthogonality conditions for epsilon} and \eqref{eq:Hardy inequ}, we get
    \begin{equation}\label{eq:bfH lower bound}
        \mathbf H \geq (\delta-C|b| )\|\epsilon\|_{\calH^{1/2}}^2-C(\epsilon, Q_b)_r^2.
    \end{equation}
    For the last term, we first note $(\epsilon,Q_b)\lesssim \|\epsilon\|_{L^2}\leq \sqrt{\alpha_0}$. From \eqref{eq:estimate from conserv}, we derive
    \begin{align}\label{eq:epsilon Q inner produce esti}
        (\epsilon, Q_b)_r^2=\calO(\sqrt{\alpha_0}(\lambda|E_0|+\norm{\epsilon}_{\mathcal{H}^{1/2}}^2 + e^{-\frac{c}{|b|}})).
    \end{align}
    Collecting \eqref{eq:b_s control pf}--\eqref{eq:epsilon Q inner produce esti}, we finish the proof.
\end{proof}
We next record the monotonicity property of the scaling parameter. 

\begin{lemma}\label{lem:lambda monotone}
    After reducing $\alpha^*>0$ if necessary, for $\alpha_0\in (0,\alpha^*)$, there is a unique $s^* \in \bbR$ such that:
    
    (1) $b(s) < 0$ for $s<s^*$, $b(s^*) = 0$, and $b(s) > 0$ for $s>s^*$.
    
    (2) Moreover, for $s_2 \geq s_1 \geq s^*$, we have
    \begin{equation}\label{eq:lambda b relation}
        \frac{1}{2}\int_{s_1}^{s_2} b(s) - C\delta(\alpha_0) \leq -\log \left(\frac{\lambda(s_2)}{\lambda(s_1)}\right) \leq \frac{3}{2}\int_{s_1}^{s_2} b(s) + C\delta(\alpha_0),
    \end{equation}
    and
    \begin{equation}\label{eq:almost monotonicity of lambda}
        \lambda(s_2) < 2\lambda(s_1).
    \end{equation}
\end{lemma}
\begin{proof}
    By taking inner product \eqref{eq:epsilon flow} with $R_1$, we obtain
    \begin{equation}\label{eq:lambda_s esti pf 1}
        \begin{aligned}
            &\left(\tfrac{\lambda_s}{\lambda}+b\right)(\Lambda Q,R_1)_r
            -\partial_s(\epsilon,R_1)_r
            \\
            &=
            -\left(\tfrac{\lambda_s}{\lambda}+b\right)(\Lambda(Q_b-Q+\epsilon),R_1)_r
            +b_s(\partial_bQ_b,R_1)_r
            +\td\gamma_s (Q_b+\epsilon,iR_1)_r
            \\
            &\quad\,
            +(\epsilon, L_{Q_b}(iR_1))_r-b(\epsilon,\Lambda R_1)_r-(i \mathcal{N}(\epsilon) -i\Psi_b,R_1)_r.
        \end{aligned}
    \end{equation}
    Combining \eqref{eq:first law}, \eqref{eq:Qb_to_Q_bound} with $X=\jp{y}^{-2}L^\infty$, \eqref{eq:dQb_to_iR1_bound}, \eqref{eq:Hardy inequ}, and \eqref{eq:b_s coercivity}, the second line of \eqref{eq:lambda_s esti pf 1} can be estimated as follows:
    \begin{equation*}
        \text{Second line of \eqref{eq:lambda_s esti pf 1}}
        \lesssim 
        (\norm{\epsilon}_{\mathcal{H}^{1/2}} + e^{-\frac{c}{|b|}})(\norm{\epsilon}_{\mathcal{H}^{1/2}} + |b|)
        \lesssim b_s + |b|^2.
    \end{equation*}
    For the first term of the third line of \eqref{eq:lambda_s esti pf 1}, by \eqref{eq: orthogonality conditions for epsilon}, $L_Q(iR_1)=i\Lambda Q$, \eqref{eq:Hardy inequ}, and \eqref{eq:b_s coercivity}, we get
    \begin{equation}\label{eq:lambda_s esti pf 2}
        (\epsilon, L_{Q_b}(iR_1))_r
        =(\epsilon, (L_{Q_b}-L_Q)(iR_1))_r-(\epsilon, i\Lambda (Q_b-Q))_r
        \lesssim |b|\norm{\epsilon}_{\mathcal{H}^{1/2}}\lesssim b_s + |b|^2.
    \end{equation}
    Thus, by \eqref{eq:lambda_s esti pf 2}, \eqref{eq:Hardy inequ}, \eqref{eq:b_s coercivity}, and \eqref{eq:Psi_exp_small_H1_conclusion}, we deduce
    \begin{align*}
        \text{Third line of \eqref{eq:lambda_s esti pf 1}}\lesssim b_s + |b|^2.
    \end{align*}
    Thus, we conclude
    \begin{equation}\label{eq:innerproduct with R_1}
        \left|\left(\tfrac{\lambda_s}{\lambda}+b\right)(\Lambda Q,R_1)_r-\partial_s(\epsilon,R_1)_r\right|
        \lesssim b_s + |b|^2.
    \end{equation}
     Since $e_1= (R_1,\Lambda Q)_r> 0$, by integrating \eqref{eq:innerproduct with R_1} with respect to $s$ and using \eqref{eq:mod para smallness}, we obtain
    \begin{equation}\label{eq:monotonicity ineq 1}
        \left|\log\left(\frac{\lambda(s_2)}{\lambda(s_1)} \right) + \int_{s_1}^{s_2} b\right| \lesssim \delta(\alpha_0)\left(1 + \int_{s_1}^{s_2}|b|\right).
    \end{equation}
    
    Now, we show (1). This is a direct consequence of negative-energy blow-up proven in \cite{Park2026arXiv}. Assume that $b(s)<0$ for all $s$. Then, from \eqref{eq:monotonicity ineq 1}, we get $-\log(\frac{\lambda(s_2)}{\lambda(s_1)})\leq C\delta(\alpha_0)+ (1-C\delta(\alpha_0))\int_{s_1}^{s_2} b$. Taking $s_1=0$ and $s_2\to \infty$, we deduce a contradiction with $\lambda(s_2)\to 0$. Thus, $b(s)$ cannot be negative for all time. Similarly, by using \eqref{eq:monotonicity ineq 1} and $\lim_{s_1\to -\infty}\lambda(s_1)=0$, $b(s)$ cannot be positive for all time. Therefore, there exists at least one $s^*$ so that $b(s^*)=0$.

    Next, we claim that for any $s^*$ so that $b(s^*)=0$, we have $b_s(s^*)>0$. This claim proves (1). Suppose that there exists a $s^*$ such that $b_s(s^*)\leq 0$. Then, since $b(s^*)=0$, \eqref{eq:b_s coercivity} implies $\norm{\epsilon(s^*)}_{\mathcal{H}^{1/2}}=0$, and thus $\epsilon(s^*)\equiv 0$. Hence, from \eqref{eq:estimate from conserv} with $b(s^*)=0$, we obtain $\lambda(s^*)=0$, which yields a contradiction. Thus we conclude (1).

    We prove (2). By (1), \eqref{eq:monotonicity ineq 1} implies \eqref{eq:lambda b relation}. To show, \eqref{eq:almost monotonicity of lambda}, we use a contradiction argument. Assume $\lambda(s_2)\geq 2\lambda(s_1)$ for some $s_2\geq s_1\geq s^*$. Then, we have
    \begin{align*}
        \log 2- C\delta(\alpha_0)\leq \log \tfrac{\lambda(s_2)}{\lambda(s_1)}- C\delta(\alpha_0)\leq 
        -\tfrac12 {\textstyle\int_{s_1}^{s_2}} b(s)<0,
    \end{align*}
    which is a contradiction for $\alpha^*$ small enough. Therefore, we conclude \eqref{eq:almost monotonicity of lambda}.
\end{proof}

It remains to convert the monotonicity estimate for $\lambda$ into the stated blow-up rate.

\begin{proof}[Finish of the proof of Theorem~\ref{thm:main thm}] The rest of argument closely follows that of Merle--Raphaël \cite{MerleRaphael2003GAFA}. 
    First, by translating the renormalized time variable $s$, we may assume $s^*=0$ in Lemma~\ref{lem:lambda monotone}.
    We first claim that there exists a large $\td s^*>0$ such that, for $s>\td s^*$,
    \begin{align}\label{eq:loglog lambda b}
        (\log|\log \lambda(s)|)^{-1}\lesssim b(s).
    \end{align}
    By \eqref{eq:b_s coercivity}, we have
    \begin{equation*}
        b_s\geq -C e^{-\frac{c}{b}}\geq -\frac 2c b^2 e^{-\frac{c}{2b}}, \quad \text{or} \quad -\frac{cb_s}{2b^2}e^{\frac{c}{2b}}\leq 1.
    \end{equation*}
    Since $b(s)>0$ for $s>0$, by taking $\td s_1 > -1 +e^{\frac{c}{2b(1)}}$, we obtain
    \begin{equation*}
        e^{\frac{c}{2b(s)}}\leq s-1 +e^{\frac{c}{2b(1)}} \leq 2 s,\quad \text{for} \quad s\geq \td s_1.
    \end{equation*}
    Hence, we have, for $s\geq \td s_1$,
    \begin{equation}\label{eq:b log bound}
        (\log(s))^{-1}\lesssim b(s).
    \end{equation}
    Since $\lambda\to 0$ as $s\to \infty$, we can take $\td s_2$ large enough so that $ \log \lambda(\td s_1) +C\delta(\alpha_0)<-\log \lambda(\td s_2)$. Thus, for $s\geq \td s_2$, \eqref{eq:lambda b relation} gives
    \begin{equation*}
        \tfrac{1}{2}{\textstyle\int_{\td s_1}^{s}} b 
        \leq 
        -\log  \lambda(s)+ \log \lambda(\td s_1) +C\delta(\alpha_0) \leq -2 \log \lambda(s).
    \end{equation*}
    Therefore, after taking $\td s_3\gg \td s_1, \td s_2$, for $s\geq \td s_3$, \eqref{eq:b log bound} yields
    \begin{equation*}
        s(\log s)^{-1} \lesssim {\textstyle\int_{\td s_1}^{s}} (\log \tau)^{-1} d\tau
        \lesssim {\textstyle\int_{\td s_1}^{s}} b 
        \lesssim -\log \lambda (s).
    \end{equation*}
    Hence, for large $s$, we get
    \begin{equation*}
        \log|\log \lambda (s)|\geq \log s-\log(\log s)\geq \tfrac12 \log s.
    \end{equation*}
    Combining this and \eqref{eq:b log bound}, we conclude \eqref{eq:loglog lambda b}.

    Now, we finish the proof. Choose $t_n$ recursively as the first time after $t_{n-1}$ such that $\lambda(t_n)=2^{-n}$, and set $s_n=s(t_n)$. For $n$ large, \eqref{eq:almost monotonicity of lambda} gives
    \begin{align*}
        2^{-(n+1)} \le \lambda(s) \le 2^{-(n-1)} \quad \text{for}\quad s \in [s_n,s_{n+1}].
    \end{align*}
    By \eqref{eq:loglog lambda b} and \eqref{eq:lambda b relation},
    \begin{align}\label{eq: loglog int bound}
        {\textstyle\int_{s_n}^{s_{n+1}}} (\log|\log\lambda|)^{-1}
        \lesssim {\textstyle\int_{s_n}^{s_{n+1}}} b
        \lesssim \delta(\alpha_0) + |\log (\tfrac{\lambda(s_{n+1})}{\lambda(s_{n})})| \lesssim 1.
    \end{align}
    Since $\frac{ds}{dt}=\frac{1}{\lambda(t)}$ and $\lambda(t) \sim \lambda(t_n)$ on $[t_n,t_{n+1}]$, \eqref{eq: loglog int bound} implies
    \begin{align*}
        \frac{t_{n+1}-t_n}{\lambda(t_n)\log|\log(\lambda(t_n))|}
        \lesssim \int_{t_{n}}^{t_{n+1}}\frac{dt}{\lambda(t)\log|\log\lambda(t)|}\lesssim 1.
    \end{align*}
    Thus, we arrive at
    \begin{align*}
        t_{n+1}-t_n
        \lesssim \lambda(t_n)\log|\log(\lambda(t_n))|
        \sim 2^{-n}\log n.
    \end{align*}
    Summing this, we arrive at
    \begin{align*}
        T-t_n \lesssim 2^{-n}\log n 
        \sim \lambda(t_n)\log|\log(\lambda(t_n))|.
    \end{align*}
    Using again \eqref{eq:almost monotonicity of lambda}, the same estimate holds for every $t$ close to $T$:
    \begin{align*}
        T-t \lesssim \lambda(t)\log|\log(\lambda(t))|.
    \end{align*}
    Since $f(x)=x\ln|\ln x|$ is increasing for $x>0$ sufficiently small, by using \eqref{eq:mod para smallness}, we conclude
    \begin{equation*}
        \|u(t)\|_{\dot H^{1/2}}^{-1}\sim\lambda^{1/2}(t) \gtrsim \sqrt{\frac{T-t}{\log|\log(T-t)|}}. \qedhere
    \end{equation*}
\end{proof}

\appendix

\section{Decomposition}\label{sec:decom append}
In this section, we provide the sketch of the proof of Lemma~\ref{lem:decomposition}. 
\begin{proof}[Proof of Lemma~\ref{lem:decomposition}]
    We first show the tube stability: for some function $\delta_{\mathrm t, 1}(\alpha)$ with $\delta_{\mathrm t,1}(\alpha)\to 0$ as $\alpha\to 0$, we have
    \begin{equation}\label{eq:tube stability aux}
        \|Q-\lambda_0(t)^{1/2}e^{i\theta(t)}u(t,\lambda_0(t)y)\|_{H^{1/2}}\leq \delta_{\mathrm t,1}(\alpha_0), 
    \end{equation}
    where $\lambda_0(t) \coloneqq \|Q\|_{\dot H^{1/2}}/\|u(t)\|_{\dot H^{1/2}}$. Denote $\td{u}(y)\coloneqq \lambda_0(t)^{1/2}u(t,\lambda_0(t)y)$. Then, $\|\td u\|_{\dot H^{1/2}} = \|Q\|_{\dot H^{1/2}}$.
    Also, by conservation of mass and energy, we have
    \begin{align*}
    \|\td u\|_{L^2}^2=\|Q\|_{L^2}^2+\alpha_0, \qquad 
    E(\td u)=\lambda_0(t)E_0<0.
    \end{align*}
    Assume that \eqref{eq:tube stability aux} fails. Then, there exists a sequence $\{u_n\}$ and $\eps>0$ such that
    \begin{equation}\label{eq:tube subseq}
        \begin{gathered}
            \|\td u_n\|_{\dot H^{1/2}} = \|Q\|_{\dot H^{1/2}},\quad
            \|\td u_n\|_{L^2}^2\to \|Q\|_{L^2}^2, \quad
            E(\td u_n)\to 0,
            \\
            \inf_{\theta\in \mathbb{R}}\|e^{i\theta}\td u-Q\|_{H^{1/2}}\geq \eps.
        \end{gathered}
    \end{equation}
    Moreover, by \eqref{eq:tube subseq} and Sobolev embeddings, we have $\|\td u_n\|_{L^4}\geq C>0$ and $\|\td u_n\|_{L^2}+\|\td u_n\|_{L^6}\leq C$. Thus, there exist $\eps_1, \eps_2>0$ such that $|\{x\in \bbR : |\td u_n(x)|>\eps_1\}|>\eps_2$. By Lieb’s compactness lemma, after passing to a subsequence, there exists a sequence of translations $\{x_n\}$ such that $\td u_n(\cdot +x_n) \rightharpoonup \td u$ a.e. and weakly in $H^{1/2}$ for some $\td u \in H^{1/2}$ with $\td u\neq 0$. Again, by \eqref{eq:tube subseq}, $\|\td u_n\|_{L^4}^4/\|\td u_n\|_{L^2}^2\|\td u_n\|_{\dot H^{1/2}}^2 \to 2/\|Q\|_{L^2}^2$, which implies that $\{u_n\}$ is an extremizing sequence of the sharp Gagliardo--Nirenberg inequality. By Brezis--Lieb Lemma with a standard argument, we conclude $\td u_n(\cdot +x_n) \to \td u$ strongly in $H^{1/2}$. Therefore, by the variational characterization of $Q$ in \cite{FrankLenzmann2013Acta} and the even condition, we deduce $\td u_n \to Q$, which yields a contradiction. Therefore, we finish to show \eqref{eq:tube stability aux}. 
    
    Now, we need to show that there exists $\gamma_0(t)\in C^1$ satisfying
    \begin{equation}\label{eq:tube stability}
        \|Q-\lambda_0(t)^{1/2}e^{i\gamma_0(t)}u(t,\lambda_0(t)y)\|_{H^{1/2}}\leq \delta_{\mathrm t,1}(\alpha_0).
    \end{equation}
    Let $\gamma_0(t)=-\mathrm{arg}(\td u (t), Q)_c$. That is, we have $e^{i\gamma_0}(\td u , Q)_c=|(\td u , Q)_c|$. Since we have $|(\td u , Q)_c|=|(e^{i\theta}\td u , Q)_c|\geq \|Q\|_{L^2}^2-\delta_{\mathrm t,1}(\alpha_0)\|Q\|_{L^2}>0$, $\gamma_0(t)$ is well-defined. In addition, $\gamma_0(t)\in C^1$ by definition. Moreover, by the definition of $\gamma_0$, we have $(e^{i\gamma_0}\td u , iQ)_r=0$. Therefore, by using this and \eqref{eq:tube stability aux},  we have
    \begin{equation*}
        0=(e^{i(\gamma_0-\theta)}e^{i\theta}\td u , iQ)_r=(e^{i(\gamma_0-\theta)}Q,iQ)_r+O(\delta_{\mathrm t,1}),
    \end{equation*}
    which yields $|\gamma_0-\theta|\lesssim \delta_{\mathrm t,1}$. Therefore, we get \eqref{eq:tube stability}, possibly after shrinking $\delta_{\mathrm t,1}(\cdot)$.

    Now, we prove the decomposition for functions satisfying the  tube stability.
    For $\delta_{\mathrm t,2}>0$, define $U_{\delta_{\mathrm t,2}} \coloneqq \{u\in H^1 : \norm{u-Q}_{H^1} \leq \delta_{\mathrm t,2}\}$. Also, for parameters $\lambda_1>0,\, \gamma_1, b\in \bbR$, define
    \[
        \epsilon_{\lambda_1,\gamma_1,b} \coloneqq e^{i\gamma_1}\lambda_1^{1/2}u(\lambda_1y) - Q_{b}(y),
    \]
    and 
    \begin{equation*}
        \mathbf{F}(\lambda_1,\gamma_1,b;u)\coloneqq 
        ((\epsilon_{\lambda_1,\gamma_1,b},\Lambda Q_b)_r, (\epsilon_{\lambda_1,\gamma_1,b},i\Lambda Q_b)_r, (\epsilon_{\lambda_1,\gamma_1,b},i\Lambda^2 Q_b)_r)^{\mathrm T}.
    \end{equation*}
    Thus, we obtain
    \begin{align*}
        \partial_{(\lambda_1,\gamma_1,b)}\mathbf{F}(1,0,0;Q)=
        \begin{pmatrix}
            \|\Lambda Q\|_{L^2}^2 & 0 & 0 \\
            0 & 0 & (R_1,i^{-1}L_Q[iR_1])_r \\
            0 & -\|\Lambda Q\|_{L^2}^2 & -(R_1,\Lambda^2Q)_r
        \end{pmatrix}.
    \end{align*}
    In particular, $\det \partial_{(\lambda_1,\gamma_1,b)}F(1,0,0;Q)\neq 0$.
    Hence, by the implicit function theorem, after fixing $\delta_{\mathrm t,2}>0$
    small enough, there exists a unique $C^1$ map $(\lambda_1,\gamma_1,b)\colon U_{\delta_{\mathrm t,2}}\to \bbR_+ \times \bbR^2$ such that
    \begin{equation}\label{eq:IFT}
        F(\lambda_1(u),\gamma_1(u),b(u);u)=0
    \end{equation}
    for all $u\in U_\delta$. Moreover,
    \begin{equation}\label{eq:IFT smallness}
        |\lambda_1(u)-1|+|\gamma_1(u)|+|b(u)|+\|\epsilon_{\lambda_1,\gamma_1,b}\|_{H^{1/2}}
        \lesssim \|u-Q\|_{H^{1/2}}.
    \end{equation}
    Hence, taking $\alpha^*$ sufficiently small so that $\delta_{\mathrm t,1}(\alpha^*)<\delta_{\mathrm t,2}$, we can apply the decomposition by \eqref{eq:tube stability}. Thus, thanks to \eqref{eq:IFT} and \eqref{eq:IFT smallness}, we finish the proof.
\end{proof}

\section{Comparison with the Merle--Rapha\"el framework}\label{sec:comparison NLS}

In this appendix, we compare our profile construction with the radial profile construction of Merle--Raphaël for the mass-critical NLS \cite{MerleRaphael2003GAFA,MerleRaphael2004Invent,MerleRaphael2005AnnMath,MerleRaphael2005CMP,Raphael2005MathAnnalen,MerleRaphael2006JAMS}. Our aim is to clarify the relation between the two constructions and to explain in what sense the present profile construction is connected to the Merle--Rapha\"el framework.

The comparison is at the level of the formal profile before localization. In the NLS construction, the pseudo-conformal phase reduces the profile equation to a real amplitude equation. We use this reduction only for the formal coefficient comparison below; the localized NLS profile and the associated radiation term are outside the scope of this appendix. On our side, the objects being compared are the formal correctors $R_j$, the truncated profiles $Q_b^{(N)}$, and the profile $Q_b$ defined by Borel's integral summation method.

We write the calculation for \eqref{eq:NLS} with the radial symmetry in dimension $d\ge1$.
Set
\begin{equation}\label{eq:appendix_profile_error_nls}
    \mathcal F_b(f)\coloneqq -\Delta f+f-|f|^{4/d}f-ib\Lambda f,
    \qquad
    \Lambda=\tfrac d2+y\cdot\nabla .
\end{equation}
Let $\mathcal Q$ be the positive radial ground state
\[
    -\Delta\mathcal Q+\mathcal Q-\mathcal Q^{1+4/d}=0.
\]
The linearized operator is
\begin{equation}\label{eq:appendix_LQ_nls}
    \mathcal L_{\mathcal Q}h
    =
    -\Delta h+h-\mathcal Q^{4/d}h-\tfrac4d\Re(\mathcal Q^{4/d} h) .
\end{equation}
The mass-critical identity is
\begin{equation}\label{eq:appendix_mass_identity_nls}
    (\mathcal F_b(f),if )_r=0.
\end{equation}
This is exactly the NLS analogue of the solvability identity used in \eqref{eq:Psi cancel}.

If one formally expands the unlocalized self-similar profile as
\[
    \mathcal Q_b^{(N)}=\mathcal Q+\sum_{j=1}^N(ib)^j\mathcal R_j,
\]
then coefficient comparison leads, at order $j$, to an equation of the form
\begin{equation}\label{eq:appendix_Rj_recursion_nls}
    \mathcal L_{\mathcal Q}[i^j\mathcal R_j]
    =i^j\Lambda\mathcal R_{j-1}+(\text{terms determined by }\mathcal R_1,\ldots,\mathcal R_{j-1}),
    \quad \mathcal R_0=\mathcal Q .
\end{equation}
The exact expression of the lower-order nonlinear terms is not used below. The solvability condition is the NLS analogue of \eqref{eq:Psi cancel}. We make no claim here about convergence of this formal expansion.

The classical NLS construction has an additional reduction which is not available for the half-wave equation \eqref{eq:half-wave}. Writing, with $r=|y|$,
\begin{equation}\label{eq:appendix_phase_reduction}
    \mathcal Q_b(y)=e^{-ibr^2/4}\mathcal P_b(r),
\end{equation}
removes the imaginary transport term and transforms the formal profile equation into the real amplitude equation
\begin{equation}\label{eq:appendix_amplitude_equation}
    -\partial_{rr}\mathcal P_b -\tfrac{d-1}{r}\partial_r\mathcal P_b 
    +(1-\tfrac{1}{4}b^2r^2)\mathcal P_b-\mathcal P_b^{1+4/d}=0.
\end{equation}
In the Merle--Rapha\"el construction, the localized profile is built inside the elliptic region of this equation and then cut off near the degeneracy radius $r=\frac{2}{|b|}$.
This is the NLS cut-off radius in physical space.

The relation between the amplitude expansion and the expansion in powers of $b$ is an algebraic coefficient identity. Suppose that the unlocalized amplitude is expanded formally:
\[
    \mathcal P_b= \sum_{k\ge0}b^{2k}\mathcal T_k, \qquad \mathcal T_0=\mathcal Q .
\]
Expanding \eqref{eq:appendix_phase_reduction} gives
\begin{equation}\label{eq:appendix_amplitude}
    \mathcal Q_b(y)=e^{-ibr^2/4}\sum_{k\ge0}b^{2k}\mathcal T_k
    =\sum_{m,k\ge0}b^{2k}\frac{(-ibr^2/4)^m}{m!}\mathcal T_k .
\end{equation}
Comparing the coefficient of $b^j$ with $\sum_j (ib)^j\mathcal R_j$ yields, for every fixed $j$,
\begin{equation}\label{eq:appendix_finite_jet_identity}
    \mathcal R_j
    =
    \sum_{k=0}^{\lfloor j/2\rfloor}
    (-1)^{j-k}
    \frac{r^{2(j-2k)}}{4^{j-2k}(j-2k)!}\mathcal T_k .
\end{equation}
This formal identity says that the pseudo-conformal amplitude expansion and the direct corrector expansion, when both are viewed before localization, give the same coefficient of each fixed power of $b$.

We now identify the corresponding order scale.
For the amplitude coefficients, \eqref{eq:appendix_LQ_nls} reduces on real-valued radial functions to
\begin{equation*}
    \mathcal L_{\mathcal Q}h=
    -\partial_{rr}h-\tfrac{d-1}{r}\partial_rh+h-(1+\tfrac4d)\mathcal Q^{4/d}h .
\end{equation*}
Coefficient comparison in \eqref{eq:appendix_amplitude_equation} gives, for $k\ge1$,
\begin{equation}\label{eq:appendix_Tk_recursion}
    \mathcal L_{\mathcal Q}\mathcal T_k=
    \tfrac{r^2}{4}\mathcal T_{k-1}+(\text{terms determined by }\mathcal T_1,\ldots,\mathcal T_{k-1}),
\end{equation}
Since $\partial_r\mathcal Q=-\mathcal Q+O(r^{-1}\mathcal Q)$ as $r\to\infty$, for every fixed $M$,
\begin{equation}\label{eq:appendix_LQ_tail}
    \mathcal L_{\mathcal Q}(r^M\mathcal Q)
    =
    2M r^{M-1}\mathcal Q+O(r^{M-2}\mathcal Q)+O(r^M\mathcal Q^{1+4/d}).
\end{equation}
The nonlinear term in \eqref{eq:appendix_Tk_recursion} is exponentially smaller than the leading $\mathcal Q$-tail. Therefore a fixed-order induction gives
\begin{equation}\label{eq:appendix_Tk_tail}
    \mathcal T_k(r)=\tfrac{1}{24^k k!}r^{3k}\mathcal Q(r)+O_k(r^{3k-1}\mathcal Q(r)),
    \qquad r\to\infty .
\end{equation}
Using \eqref{eq:appendix_amplitude} and the leading part of \eqref{eq:appendix_Tk_tail}, its $\mathcal Q$-tail has the formal size
\begin{equation}\label{eq:appendix_mk_size}
    \mathcal Q_b(y)
    =\sum_{m,k\ge0}\frac{1}{m!k!}\bigg(\frac{|b|r^2}{4}\bigg)^m\bigg(\frac{b^2r^3}{24}\bigg)^k \mathcal Q(r).
\end{equation}
By Stirling's formula, the quantity in \eqref{eq:appendix_mk_size} is largest when
\begin{equation}\label{eq:appendix_relevant_indices}
    m\sim \tfrac{1}{4}|b|r^2,
    \qquad
    k\sim \tfrac{1}{24}b^2r^3.
\end{equation}
Since the direct corrector order is $j=m+2k$, this gives the scale relation
\begin{equation}\label{eq:appendix_effective_order}
    j \sim |b|r^2+b^2r^3 .
\end{equation}
Thus the NLS degeneracy radius $r\sim |b|^{-1}$ corresponds to the total order $j\sim |b|^{-1}$.

We now compare this with the Borel summation used in our construction.
Our construction replaces this sharp cutoff for $j$ by the incomplete-gamma weight in \eqref{eq:Qb_tr_gamma_series}.
Indeed, inserting \eqref{eq:appendix_finite_jet_identity} into \eqref{eq:Qb_tr_gamma_series} gives
\begin{equation}\label{eq:appendix_diagonal_borel_formula}
    \mathcal Q_b=\mathcal Q+\sum_{\substack{k,m\ge0\\m+2k\ge1}}\frac{\Gamma(m+2k+1,L)}{(m+2k)!}b^{2k}\frac{(-ibr^2/4)^m}{m!}\mathcal T_k .
\end{equation}
The factor $\frac{\Gamma(m+2k+1,L)}{(m+2k)!}$ changes when $m+2k$ is comparable to $L$.
Starting from the Merle--Rapha\"el side, the spatial cutoff near $r\sim |b|^{-1}$ selects, through \eqref{eq:appendix_effective_order}, the total order $j=m+2k\sim |b|^{-1}$ in \eqref{eq:appendix_diagonal_borel_formula}.
Moreover, the incomplete-gamma weight gives the sharp truncation order $L(b)\sim N(b)\sim |b|^{-1}$.
Conversely, starting from our construction, the choice $N(b)\sim |b|^{-1}$, implemented smoothly by $L(b)\sim |b|^{-1}$, gives the same transition; using the phase reduction \eqref{eq:appendix_phase_reduction}, \eqref{eq:appendix_effective_order} places this transition at the spatial scale $r\sim |b|^{-1}$.
This is the sense in which the Merle--Rapha\"el profile construction and our construction use the same formal cutoff: the former realizes it as a spatial cutoff after the pseudo-conformal reduction, whereas the latter realizes it as an order cutoff, smoothed by the parameter $L$.

 \bibliographystyle{abbrv}
\bibliography{reference}

\end{document}